\documentclass[]{article}
\usepackage{graphicx}
\usepackage{caption}
\usepackage{algorithmic}
\usepackage{algorithm}
\usepackage{amsmath}
\usepackage{amssymb}
\usepackage{amsfonts}
\usepackage{amsthm}
\usepackage{fullpage}
\usepackage{url}
\usepackage{color}
\usepackage{authblk} 
\usepackage{float}
\usepackage{mathabx} 
\usepackage{subfig}

\def\N{{\mathbb{N}}}
\def\R{{\mathbb{R}}}

\def\dt{{\textrm{d}t}}

\def\d{{\textrm{d}}}

\def\IMF{{\textrm{IMF}}}

\newtheorem{theorem}{Theorem}

\newtheorem{definition}{Definition}

\begin{document}

\title{Multivariate Fast Iterative Filtering for the decomposition of nonstationary signals}

\author{Antonio~Cicone\thanks{DISIM, University of L'Aquila, L'Aquila, Italy ({\tt antonio.cicone@univaq.it})\\ Istituto di Astrofisica e Planetologia Spaziali, INAF, Roma, Italy\\
Istituto Nazionale di Geofisica e Vulcanologia, Roma, Italy} 

Enza~Pellegrino\thanks{DIIIE, University of L'Aquila, L'Aquila, Italy ({\tt enza.pellegrino@univaq.it})}
}

\maketitle

\begin{abstract}
In this work, we present a new technique for the decomposition of multivariate data, which we call Multivariate Fast Iterative Filtering (MvFIF) algorithm.
We study its properties, proving rigorously that it converges in finite time when applied to the decomposition of any kind of multivariate signal.

We test MvFIF performance using a wide variety of artificial and real multivariate signals, showing its ability to: separate multivariate modulated oscillations; align frequencies along different channels; produce a quasi--dyadic filterbank when decomposing white Gaussian noise; decompose the signal in a quasi--orthogonal set of components; being robust to noise perturbation, even when the number of channels is increased considerably.

Finally, we compare its performance with the one of the main methods developed so far in the literature, proving that MvFIF produces, without any a priori assumption on the signal under investigation and in a fast and reliable manner, a uniquely defined decomposition of any multivariate signal.
\end{abstract}

%

\section{Introduction}\label{sec:Intro}

The decomposition and time--frequency (TF) analysis of nonstationary signals is a long lasting line of research which have led over the decades to the development of many and important new algorithms and approaches, which are nowadays commonly used, practically, in any field of research.

Among these methods there is the so called Empirical Mode Decomposition (EMD) method proposed in 1998 by Huang and his collaborators \cite{huang1998empirical}. EMD is a local and adaptive completely data--driven method which has an iterative ``divide et impera'' approach. The idea is simple, but powerful. We first iteratively divide a given signal $s(t)$ into several simple oscillatory components $c_k(t)$ plus a trend $r(t)$
\begin{equation}\label{eq:IMFs}
    s(t)=\sum_{k=1}^{K}c_k(t)+r(t)
\end{equation}
Then each oscillatory component is analyzed separately in the TF domain via the computation of its instantaneous frequency \cite{huang2009instantaneous}. This approach allows to bypass the Heisenberg--Gabor uncertainty principle, overcoming artificial spectrum spread caused by sudden changes. EMD allows to produce each simple component via the subtraction of the signal moving average which is computed as the average between two envelopes connecting its minima and maxima. The $c_k(t)$ simple components, named Intrinsic Mode Functions (IMFs), are functions that fulfill two properties: the number of zero crossing equals the number of extrema or differ at most by one; considering upper and lower envelopes connecting respectively all the local maxima and minima of the function, their mean has to be zero at any point. For more details on the EMD and its properties we refer the interested reader to \cite{huang1998empirical}.

The EMD received a lot of attention over the years, as testified by the high number of citations of Huang's papers\footnote{The original work by Huang et al. \cite{huang1998empirical} as received so far, by itself, more than 14100 unique citations, according to Scopus}, and it has been applied to a wide variety of problems, see for instance \cite{stallone2020new} and references therein. Nevertheless, the EMD algorithm contains a number of heuristic and ad hoc elements that make hard to analyze mathematically its guarantees of accuracy or the limitations of its applicability. This is because the core of the algorithm relies heavily on interpolates of signal maxima and minima. This very approach does have also some stability problems in the presence of noise, as illustrated in \cite{wu2009ensemble}. Several variants of the EMD have been recently proposed to address this last problem, e.g. the Ensemble Empirical Mode Decomposition (EEMD) \cite{wu2009ensemble}, the complementary EEMD \cite{yeh2010complementary}, the complete EEMD \cite{torres2011complete}, the partly EEMD \cite{zheng2014partly}, the noise assisted multivariate EMD (NA-MEMD) \cite{rehman2013EMDviaMEMD}. They all allow to address this issue as well as reduce the so called mode mixing problem. But they pose new challenges both to our mathematical understanding of this kind of techniques and to the ability of these methods to handle nonstationary signals, since they worsen the mode--splitting problem present in the EMD algorithm \cite{yeh2010complementary}.

For all these reasons many research groups started working on this topic and proposed their alternative approaches to signals decomposition. We recall, for instance, the sparse time--frequency representation \cite{hou2011adaptive,hou2009variant}, the Geometric mode decomposition \cite{yu2018geometric}, the Empirical wavelet transform \cite{gilles2013empirical}, the Variational mode decomposition \cite{dragomiretskiy2013variational}, and similar techniques, like for instance \cite{meignen2007new}. All of these methods are based on optimization with respect to an a priori chosen basis.

The only alternative method proposed so far in the literature which is based on iterations, and hence does not require any a priori assumption on the signal under analysis, is the Iterative Filtering (IF) algorithm \cite{lin2009iterative}. This alternative iterative method, although published only recently, has already been used effectively in a wide variety of applied fields like, for instance, in \cite{spogli2019role,sharma2017automatic,li2018entropy,hossein2020disentangling}. The structure of the IF algorithm resemble the EMD one. Its key difference is in the way the signal moving average is computed, i.e., via correlation of the signal itself with an a priori chosen filter function, instead of using the average between two envelopes.

In particular in the IF technique the moving average $\mathcal{M}(s)(x)$ is computed as a convolution of the given signal $s$ with a compactly supported filter or window $w$ which provides the weights to be used to compute the local average
\begin{equation}\label{eq:Mov_Average}
\mathcal{M}(s)(x)=\int_{-L}^{L} s(x+t)w(t)\dt
\end{equation}
The method becomes nonlinear because $L$, which is called the \emph{filter length} and represents half support length of the filter function $w$, is computed based on the information contained in the signal itself. For instance, as suggested in \cite{lin2009iterative,cicone2016adaptive}, we can make use of the relative distance between subsequent extrema of the signal under study to determine the value $L$.

This apparently simple difference opened the doors to the mathematical analysis of IF \cite{wang2013convergence,cicone2016adaptive,huang2009convergence,cicone2020Direct}. Among the results published, we mention the ones regarding the IF a priori guarantee of convergence \cite{cicone2016adaptive}, and IF acceleration in what is called Fast Iterative Filtering (FIF), which is based on special properties of the matrices involved in the moving average whose calculation can be speed up via Fast Fourier Transform (FFT) \cite{cicone2020Direct}. This last result allowed also to transform IF into a direct technique \cite{cicone2020Direct}. All these accelerations are extremely valuable especially in view of the extension of this method to handle multidimensional and multivariate signals which may contain even billions of sample points. If, on the one hand, IF method has been already extended to deal with multidimensional signals and its convergence and stability have been already established theoretically \cite{cicone2017multidimensional}, the extension of IF algorithm to handle multivariate signals has never been addressed in the literature. However, multivariate signals are ubiquitous in real life. We can think, for instance, to physiological signals in Medicine, like the ECG, EEG, and similar data sets, or the study of complex systems done in Engineering, Finance, Economy, Physics and many other fields of research, where several quantities have to be analyzed at once. It is clear that there is the need to develop a multivariate extension of the IF method. We plan to do so in this work.

It is important to underline that, over the decades, many methods have been proposed which allow to iteratively decompose a signal into simple oscillatory components. We recall here, for instance, the MUltiple SIgnal Classification (MUSIC) algorithm \cite{schmidt1986multiple}, relaxation (RELAX) algorithm \cite{li1996efficient}, sparsity-driven optimization technique \cite{onhon2011sparsity}, and similar techniques. All these methods require an a priori basis selection on which the signal is projected and the knowledge of the number of components contained in the signal. The IF-based methods, instead, like the EMD-based techniques, do not require any a priori selection of a basis, and the a priori knowledge of the number of simple oscillatory components contained in a given signal. Moreover, IF-based methods have been widely studied from a mathematical point of view and their properties are all well known now \cite{huang2009convergence,cicone2020Direct}, whereas EMD-based techniques are still missing \cite{huang2009convergence}. 

The rest of the paper is organized as follows: in the next section we review the main methods developed so far for the decomposition of multivariate signals. In Section \ref{sec:MvFIF} we recall the FIF method and its main properties. Then we detail the newly proposed algorithm called Multivariate Fast Iterative Filtering (MvFIF) technique and we prove its convergence properties. Section \ref{sec:Performance} is devoted to the analysis of the performance of this technique. In particular, we show its ability to: separate multivariate modulated oscillations; align frequencies along different channels; produce a quasi--dyadic filterbank when decomposing white Gaussian noise; decompose the signal in a quasi--orthogonal set of components; being robust to noise perturbation, even when the number of channels is increased considerably. In Section \ref{sec:Examples} we present some applications to real life signals and we compare the MvFIF performance with the one of the main techniques available in the literature. The paper ends with Section \ref{sec:End}, where we derive our final conclusions and outline future directions of research.

\section{Multivariate extensions of nonstationary signal decomposition methods}\label{sec:review}

Several methods have been proposed in the literature which extend EMD method to handle multivariate nonstationary signals, i.e. multichannels nonstationary signals. This kind of signals require a special handling: alignment of frequencies along all the channels has to be guaranteed. The first methods proposed were able to handle complex and bidimensional nonstationary signals evolving over time, like, for instance, in  \cite{Tanaka2006ComplexEMD,rilling2007Bivariate}. Subsequently, first a trivariate, and then a multivariate generalization of the standard EMD method, called Multivariate Empirical Mode Decomposition (MEMD), were introduced in \cite{rehman2010trivariate} and \cite{rehman2009Multivariate}, respectively.

The MEMD boils down to the idea of computing the decompositions of several projections along different directions of an n--dimensional signal evolving over time and then deriving a unique average decomposition. The signal projections are chosen using a suitable set of directions obtained from a wise sampling on an n--sphere, see \cite{rehman2009Multivariate} for more details.

The MEMD algorithm, which was further developed and analyzed, for instance, in \cite{thirumalaisamy2018fast}, proved to be useful in many applications \cite{fleureau2011multivariate,park2012classification,hemakom2016adaptive}, and even to stabilize the standard EMD method \cite{rehman2013EMDviaMEMD}.

This approach, however, proves to be computationally demanding, since many different projections of the same signal are required in the computations. The problem becomes even more evident as the number of channels grows. For this reason the so called Fast and Adaptive Multivariate and Multidimensional EMD (FA-MVEMD) has been recently proposed \cite{thirumalaisamy2018fast}.
However, both MEMD and FA-MVEMD cannot produce a truly unique decomposition, because of the different projections required to produce the envelopes.

Nevertheless, given the success that MEMD and related method obtained, other research groups proposed in the last few years alternative methods for the decomposition of multivariate signals. We mention here, the Multivariate Variational Mode Decomposition (MVMD) method \cite{ur2019multivariate}, which is based on optimization, the Multivariate Singular Spectrum Analysis (MSSA) algorithm \cite{golyandina2015MSSA}, which is based on singular value decomposition, the Synchrosqueezing--based time-frequency analysis method \cite{ahrabian2015synchrosqueezing}, the empirical wavelet transform based approach for multivariate data processing and the multivariate projection-based EWT \cite{singh2018empirical,tripathy2020}, which leverage on the wavelet transform.

All these methods present a certain degree of uncertainty. As already mentioned, the MEMD, and also the multivariate projection-based EWT \cite{singh2018empirical,tripathy2020}, require to compute projections of the signal. This has two side effects: it slows down the computations, especially as the number of channels grows, and it prevents these methods from providing a uniquely identified decomposition. Moreover, EMD-based techniques are still missing a mathematical foundation \cite{huang2009convergence}. Wavelet based techniques, like the Synchrosqueezing--based time-frequency analysis method, the empirical wavelet transform based approach for multivariate data processing, and the multivariate projection-based EWT require to select a priori the basis to be used. MVMD and MSSA techniques both necessitate to set a priori the number of components to be extracted. Furthermore, MSSA requires also the a priori selection of a window size. 

In this work we propose a new technique which has an iterative structure, as MEMD, but it is now based on the iterative method called Fast Iterative Filtering technique, which has been proven recently to be convergent and stable \cite{cicone2020Direct}. MvFIF proves to be able to quickly decompose, in a uniquely determined way, a multivariate signal in simple oscillatory components in a fast and certain way, without any a priori assumption on the signal under investigation.

\section{Proposed Approach}\label{sec:MvFIF}

\subsection{Fast Iterative Filtering method}

We start this section by reviewing the Fast Iterative Filtering method and its properties.

The key idea behind this decomposition technique is the separation of simple oscillatory components contained in a signal $s$, the so called IMFs, by approximating the moving average of $s$ and iteratively subtracting it from $s$ itself.
We remind that the idea of iteratively subtracting moving averages comes from the Empirical Mode Decomposition (EMD) method \cite{huang1998empirical}, where the moving average was computed as a local average between an envelope connecting the maxima and one connecting the minima of the signal under study. The use of envelopes in an iterative way is the reason why the EMD--based algorithms are still lacking a rigorous mathematical framework.

The approximated moving average in FIF is computed, instead, by convolution of $s$ with a window/filter function $w$ \eqref{eq:Mov_Average}, as explained in \cite{lin2009iterative}. We recall that

\begin{definition}\label{def:window}
    A filter/window $w$ is a nonnegative and even function in $C^0\left([-L,\ L]\right)$, $L>0$, and such that $\int_\R w(z)\d z=\int_{-L}^{L} w(z)\d z=1$.
\end{definition}

The FIF method is detailed in the pseudo code Algorithm \ref{algo:FIF}, where DFT and iDFT stand for Discrete Fourier Transform and inverse DFT, respectively, and $\widehat{s}$ represents the Fourier transform of $s$.

\begin{algorithm}
\caption{\textbf{Fast Iterative Filtering} IMF = FIF$(s)$}\label{algo:FIF}
\begin{algorithmic}
\STATE IMF = $\left\{\right\}$
\WHILE{the number of extrema of $s$ $\geq 2$}
      \STATE  compute the filter length $L$ for $s(x)$ and the corresponding filter $w$
      \STATE  compute the DFT of the signal $s$ and of the filter $w$
      \WHILE{the stopping criterion is not satisfied}
                  \STATE  $\widehat{s}_{m+1} = (I-\textrm{diag}\left(\textrm{DFT}(w)\right))^{m}\textrm{DFT}(s)$
                  \STATE  $m = m+1$
      \ENDWHILE
      \STATE IMF = IMF$\,\cup\,  \{ \textrm{iDFT}\left(\widehat{s}_{m}\right)\}$
      \STATE $s=s-\textrm{iDFT}\left(\widehat{s}_{m}\right)$
\ENDWHILE
\STATE IMF = IMF$\,\cup\,  \{ s\}$
\end{algorithmic}
\end{algorithm}

The FIF algorithm assumes implicitly the periodicity of the signal $s$ at the boundaries, since it relies on the FFT. This is apparently an important limitation to the applicability of this technique to a generic signal. However, it is always possible to pre--extend any signal in order to make it to become periodical at the new boundaries. In \cite{stallone2020new} an algorithm for the signal pre--extension is proposed.

The FIF method contains two nested loops: an Inner Loop, and an Outer Loop.
The first IMF is computed repeating iteratively the steps contained in the inner loop until a predefined stopping criterion satisfied. Then, to produce the $2$-nd IMF, the same procedure is applied to the remainder signal $r=s-\textrm{IMF}_1$. This is what is done in the so called outer loop. Subsequent IMFs are produced iterating this last loop.

The algorithm stops when $r$ becomes a trend signal, meaning it has at most one local extremum.

A complete a priori convergence analysis of this technique has been done in \cite{cicone2020Direct}, where a theorem on the inner loop convergence and another one on the outer loop convergence have been proved.

Regarding the choice of the filter shape, following \cite{cicone2016adaptive}, and what has been done in many applications \cite{sharma2017automatic,li2018entropy,spogli2019role,stallone2020new,hossein2020disentangling}, in the following we opt for a Fokker-Planck filter, which has the nice property of being $C^\infty(\R)$ and compactly supported.

For the computation of the filter length $L$, following \cite{lin2009iterative}, we can use the following formula
\begin{equation}\label{eq:Unif_Mask_length}
L:=2\left\lfloor\xi \frac{N}{k}\right\rfloor
\end{equation}
where $N$ is the total number of sample points of the signal $s$, $k$ is the number of its extreme points, $\xi$ is a tuning parameter which needs to be tuned only when a new filter shape is selected (usually fixed around 1.6 for the Fokker-Planck filters), and $\left\lfloor \cdot \right\rfloor$ rounds a positive number to the nearest integer closer to zero. In doing so we are computing some sort of average highest frequency contained in $s$.

Another possible way could be the calculation of the Fourier spectrum of $s$ and the identification of its highest frequency peak. The filter length $L$ can be chosen to be proportional to the reciprocal of this value.

The computation of the filter length $L$ is an important step of the FIF technique. Clearly, $L$ is strictly positive and, more importantly, it is based solely on the signal itself. This last property makes the method nonlinear.

In fact, if we consider two signals $p$ and $q$ where $p\neq q$, assuming $\textrm{IMFs}(\bullet)$ represent the decomposition of a signal into IMFs by FIF, the fact that we choose the half support length based on the signal itself implies that, in general,

$$\textrm{IMFs}(p+q)\neq \textrm{IMFs}(p)+\textrm{IMFs}(q)$$

We point out that FIF algorithmic complexity is driven by the FFT calculation which is the bottle neck in this algorithm. Therefore its algorithmic complexity is $O(m\log(m))$, where $m$ is the length of the signal under study.

\subsection{Multivariate Fast Iterative Filtering method}

Given a $n$-dimensional signal evolving over time $s\in\R^n \times \R$, the idea behind the Multivariate Fast Iterative Filtering (MvFIF) algorithm is to first compute in some way a unique \emph{filter length} $L$, which represents half support length of the filter function $w$, and then use it to extract the first IMF from each of the $n$ channels separately via FIF.

Considering $s$ as a sequence of column vectors $\textbf{v}(t)=\left[v_{i}(t)\right]_{i=1,\ldots, n}$ rotating in $\R^n$, as $t$ varies in $\R$, we can compute the filter length using $\theta(t)$, angle of rotation of such vectors over time, defined as
\begin{equation}\label{eq:theta}
\theta(t)=\arccos\left(\frac{\textbf{v}(t)}{\left\|\textbf{v}(t)\right\|}\cdot\frac{ \textbf{v}(t-1)}{\left\|\textbf{v}(t-1)\right\|}\right)
\end{equation}
In particular we use as filter length $L$ the double average distance between subsequent extrema in $\theta(t)$.

We point out that this approach is very natural if we consider a multivariate IMF as a vector in $\R^n$ rotating around the time axis. Computing the double average distance between subsequent extrema in $\theta(t)$ allows to estimate the average scale of the highest frequency rotations embedded in the given signal.

If we assume that the signal $s$ is sampled over time at $m$ points, then $s$ is a matrix in $\R^{n\ \times\ m}$ and we use the notation $s=[\textbf{v}_1 \textbf{v}_2 \ldots \textbf{v}_m]$
where each $\textbf{v}_j$ is a column vector in $\R^n$, and
\begin{equation}\label{eq:s_row_vec}
    s=\left[\begin{array}{c}
               \textbf{u}_1  \\
         \textbf{u}_2  \\
         \ldots \\
          \textbf{u}_n
              \end{array}
       \right]
\end{equation}
where each $\textbf{u}_i$ is a row vector in $\R^m$.

The pseudo code of MvFIF is given in Algorithm \ref{algo:MVFIF} and a Matlab implementation is available online\footnote{\url{http://www.cicone.com}}.

\begin{algorithm}
\caption{\textbf{Multivariate Fast Iterative Filtering} IMF = MvFIF$(s)$}\label{algo:MVFIF}
\begin{algorithmic}
\STATE IMF = $\left\{\right\}$
\STATE compute $\theta(t)$ using \eqref{eq:theta}
\WHILE{the number of extrema of $\theta$ $\geq 2$}
      \STATE compute the filter length $L$ of the filter function $w$
      \STATE set $k=0$
      \WHILE{the stopping criterion is not satisfied}
            \FOR{ $i=1$ \TO $n$ }
                 \STATE  $\left(\widehat{\textbf{u}}_i^{(k)}\right)^T = \left(I-\textrm{diag}\left(\textrm{DFT}(w)\right)\right)^{k}\textrm{DFT}(\textbf{u}_i^T)$
            \ENDFOR
            \STATE  $k = k+1$
      \ENDWHILE
      \STATE IMF = IMF$\,\cup\,  \left\{ \left[\textrm{iDFT}\left(\widehat{\textbf{u}}_i^{(k)}\right)\right]_i\right\}$
      \STATE $s=s-\left[\textrm{iDFT}\left(\widehat{\textbf{u}}_i^{(k)}\right)\right]_i$
      \STATE compute $\theta(t)$ using \eqref{eq:theta}
\ENDWHILE
\STATE IMF = IMF$\,\cup\,  \{ s\}$
\end{algorithmic}
\end{algorithm}
In the pseudocode, DFT and iDFT stand for Discrete Fourier Transform and inverse DFT, respectively, whereas $\textbf{u}_i^{(k)}$ represents the value of the IMF corresponding to the $i$--th channel at the $k$--th step of the inner loop, and $\widehat{\textbf{u}}_i^{(k)}$ represents its Fourier transform.

Given $\delta>0$, we can define the following stopping criterion: $\exists N_0\in\N$ such that
\begin{equation}\label{eq:Discrete_Abs_StopCond}
\textrm{SC}=\max_{i=1,\ldots,\ n}\|\textbf{u}_i^{(k+1)}-\textbf{u}_i^{(k)}\|_2<\delta \quad \forall k\geq N_0
\end{equation}

Regarding the a priori convergence of the proposed method we are now ready to prove the following
\begin{theorem}[MvFIF convergence]
Given a signal $s\in\R^{n \times m}$, assumed to be periodical at the time boundaries, and a filter vector $w$ derived from a symmetric filter $h$ convolved with itself.
If we consider as filter length $L$ the double average distance between subsequent extrema in the function $\theta(t)$, defined in \eqref{eq:theta}, and if we
fix $\delta>0$, then, for the minimum $N_0\in\N$ such that the stopping criterion \eqref{eq:Discrete_Abs_StopCond} holds true $\forall k\geq N_0$, the first IMF of $s$ is given by
\begin{eqnarray}\label{eq:FIF_algo}
 \textrm{IMF} &=& \left[\textrm{iDFT}\left((I-D)^{N_0}\textrm{DFT}(\textbf{u}_i^T)\right)^T\right]_{i=1,\ldots,\ n}
\end{eqnarray}
where $D$ is a diagonal matrix containing as entries the eigenvalues of the discrete convolution matrix $W$ associated with the filter vector $w$.
\end{theorem}

\begin{proof}
We assume we have a signal $s\in\R^{n \times m}$, and a filter vector $w\in\R^m$, derived from a symmetric filter $h$ convolved with itself and such that its filter length $L$ equals the double average distance between subsequent extrema in the function $\theta(t)$ defined in \eqref{eq:theta}. We can construct the $W\in\R^{m\times m}$ discrete convolution operator associated with $w$. It can be a circulant matrix, Toeplitz matrix or it can have a more complex structure. Its structure depends on the way we extend the signal outside its boundaries. Since we are assuming that we have periodic extension of signals outside the time boundaries of the signal, then $W$ is a circulant matrix given by
\begin{equation}\label{eq:W}
   W=\left[
         \begin{array}{cccc}
           c_0 & c_{m-1} & \ldots & c_1 \\
           c_{1} & c_0 & \ldots & c_2 \\
           \vdots & \vdots & \ddots & \vdots \\
           c_{m-1} & c_{m-2} & \ldots & c_0 \\
         \end{array}
       \right]
\end{equation}
where $c_j\geq 0$, for every $j=0,\ldots, \ m-1$, and $\sum_{j=0}^{m-1}c_j=1$.
Each row contains a circular shift of the entries of a chosen vector filter $w$.

It is well known that this matrix is diagonalizable via a unitary matrix $U$ which contains as columns the so called Fourier Basis

\begin{eqnarray}\label{eq:Eigenvectors}
\nonumber u_p\ =\ \frac{1}{\sqrt{m}} \left[1,\ e^{-2\pi i p\frac{1}{m}},\ldots,\ e^{-2\pi i p\frac{m-1}{m}}\right]^T \\
 p\ =\ 0,\ldots,\ m-1
\end{eqnarray}
and, in particular, $W=UDU^T$ where $D$ is a diagonal matrix whose entries are the eigenvalues of $W$, which are given by
\begin{equation}\label{eq:Lambdas}
\lambda_p\ =\ \sum_{q=0}^{m-1} c_{1q}e^{-2\pi i p \frac{q}{n}}  \qquad  \qquad p\ =\ 0,\ldots,\ m-1
\end{equation}
that are equivalent to $\textrm{DFT}(w)$, in the problem under study.

Furthermore, given the hypotheses on $w$, it follows that every $\lambda_p$ is contained in the positive interval $[0,1]$.

Therefore, for every fixed $i=1,\ldots,\ n$, it follows that
\begin{eqnarray}
 \nonumber  \textbf{u}_i^{(k)}-\textbf{u}_i^{(k+1)} = (I-W)^{k}\textbf{u}_i-(I-W)^{k+1}\textbf{u}_i = & \\
 \nonumber U(I-D)^{k}(I-(I-D))U^T\textbf{u}_i= &\\
  UD(I-D)^{k}\widetilde{\textbf{u}_i} \longrightarrow \textbf{0} \quad \textrm{ as } m\longrightarrow \infty&
\end{eqnarray}
where $\widetilde{\textbf{u}_i}=U^T \textbf{u}_i$.

In particular we have that $\|UD(I-D)^{k}\widetilde{\textbf{u}_i}\|_2$ decreases monotonically to 0 and hence
for every fix $\delta>0$, $\exists$~$N_i\in\N$ such that $\|\textbf{u}_i^{(k)}-\textbf{u}_i^{(k+1)}\|_2<\delta$, $\forall k\geq N_i$.

If we define $N_0=\max_{i=1,\ldots,\ n} N_i$, then the stopping criterion \eqref{eq:Discrete_Abs_StopCond} holds true $\forall k\geq N_0$, and the conclusion of the theorem follows directly.
\end{proof}
We recall that computing the DFT of a vector is equivalent to multiplying it on the left by $U^T$, whereas the iDFT is equivalent to multiplying it on the left by~$U$.

The MvFIF approach has several advantages with respect to previously developed methods proposed in the literature.

First of all there is no need to make any assumption on the signal under investigation. This is required, instead, in methods based on optimization like, for instance, the Multivariate Variational Mode Decomposition (MVMD) method \cite{ur2019multivariate}, or methods based on wavelets, e.g.  the Synchrosqueezing--based time-frequency analysis method \cite{ahrabian2015synchrosqueezing} and the empirical wavelet transform based approach for multivariate data processing, like the multivariate projection-based EWT \cite{singh2018empirical,tripathy2020}, where a basis has to be selected a priori. 

Moreover, MvFIF, as for EMD based methods, does not require to specify a priori the number of IMFs to be produced by the method. Whereas, methods like MVMD,  MSSA, and the more recent multivariate sliding mode SSA \cite{jain2020}, require this number to be set before running the decomposition.

Furthermore the proposed method has no limitations on the number of channels, like in the bivariate and trivariate extension of EMD proposed in the past.

This new algorithm produces directly a uniquely defined decomposition, whereas the MEMD \cite{rehman2009Multivariate} and its variants produce the final decomposition as the average among all the decompositions of many wisely chosen different projections of the original signal. 

Another important aspect regards its computational complexity. As we already mentioned, FIF computational complexity is $O(m\log(m))$, where $m$ is the length of the single channel signal under study. This is due to the acceleration of the Iterative Filtering moving average calculation via FFT, as detailed in \cite{cicone2020Direct}. Therefore, MvFIF computational complexity is $O(n m\log(m))$, where $n$ is the number of channels and $m$ is the number of time samples of the signal. This makes the proposed approach extremely fast and competitive with respect to any other method proposed so far in the literature, as shown also in the following numerical examples.

Finally, since MvFIF is a direct extension of FIF algorithm for multivariate data, it inherits all the important advantages and nice properties of the original FIF. Such as its fast convergence as well as its stability to noise, which have been extensively studied in \cite{cicone2016adaptive,cicone2020Direct}. In the rest of this work we will show and detail the performance of MvFIF by means of numerical examples.

\section{Performance of Multivariate Fast Iterative Filtering algorithm}\label{sec:Performance}

To evaluate the performance of the proposed method, we report in the following detailed results of simulations and experiments we conducted on a wide variety of multivariate signals.

We will first of all show the ability of MvFIF\footnote{MvFIF code is available at \url{www.cicone.com}} to separate multivariate modulated oscillations and to align common frequency scales across multiple data channels. These are fundamental features if we want to study real life applications involving multivariate signals.
In particular, we test the method against a real world bivariate data and show the ability of this method in producing multivariate modulated oscillations.
Furthermore we test it against a synthetic signal containing a combination of nonstationary components across different channels. We compare this result with those obtained applying the original FIF to each channel separately.

Subsequently, we show the effect of noise in the signal on mode--alignment in the MvFIF decomposition, comparing its performance with the one of the MEMD\footnote{MEMD code is available at \url{https://www.commsp.ee.ic.ac.uk/~mandic/research/emd.htm}} algorithm.

Finally, we test the proposed method on real world data sets which include EEG multichannel signals and the Earth magnetic field measurements, and we compare the performance of MvFIF with the one of the main methods proposed so far in the literature.

All the following tests have been conducted on a laptop equipped with an Intel Core i7-8550U CPU, 1.80GHz, 16.0 GB RAM, Windows 10 Pro, Matlab R2020a.

\subsection{Separation of Multivariate Modulated Oscillations}

We start by demonstrating the ability of the proposed method to separate multivariate modulated oscillations from a real bivariate data set. To this aim, we consider the multichannel signal analyzed in \cite{ur2019multivariate} which consists of position records, East and North displacement in kilometers, of a subsurface oceanographic float that was deployed in North Atlantic ocean to track the trajectory of salty water flowing from the Mediterranean Sea. The measurements are part of the Eastern Basin experiment stored in the World Ocean Circulation Experiment Subsurface Float Data Assembly Center (WFDAC) which can be downloaded at \url{http://wfdac.whoi.edu}.

\begin{figure}%
    \centering
    \subfloat{{\includegraphics[width=0.5\textwidth]{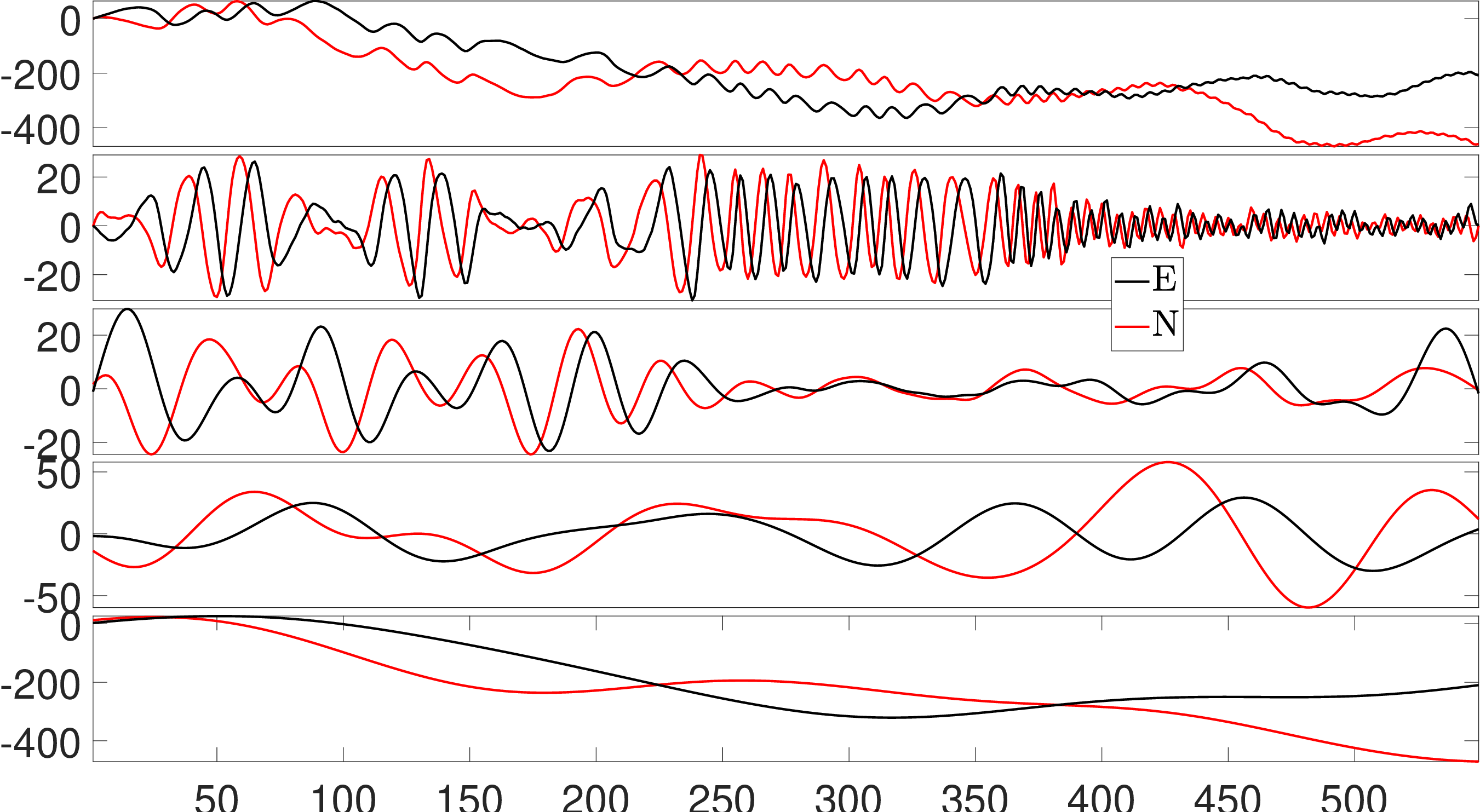} }}\\
    \caption{MvFIF decomposition of the subsurface float data measured in km}\label{fig:Ex1_IMFs}
\end{figure}

\begin{figure}%
    \centering
    \subfloat{{\includegraphics[width=0.24\textwidth]{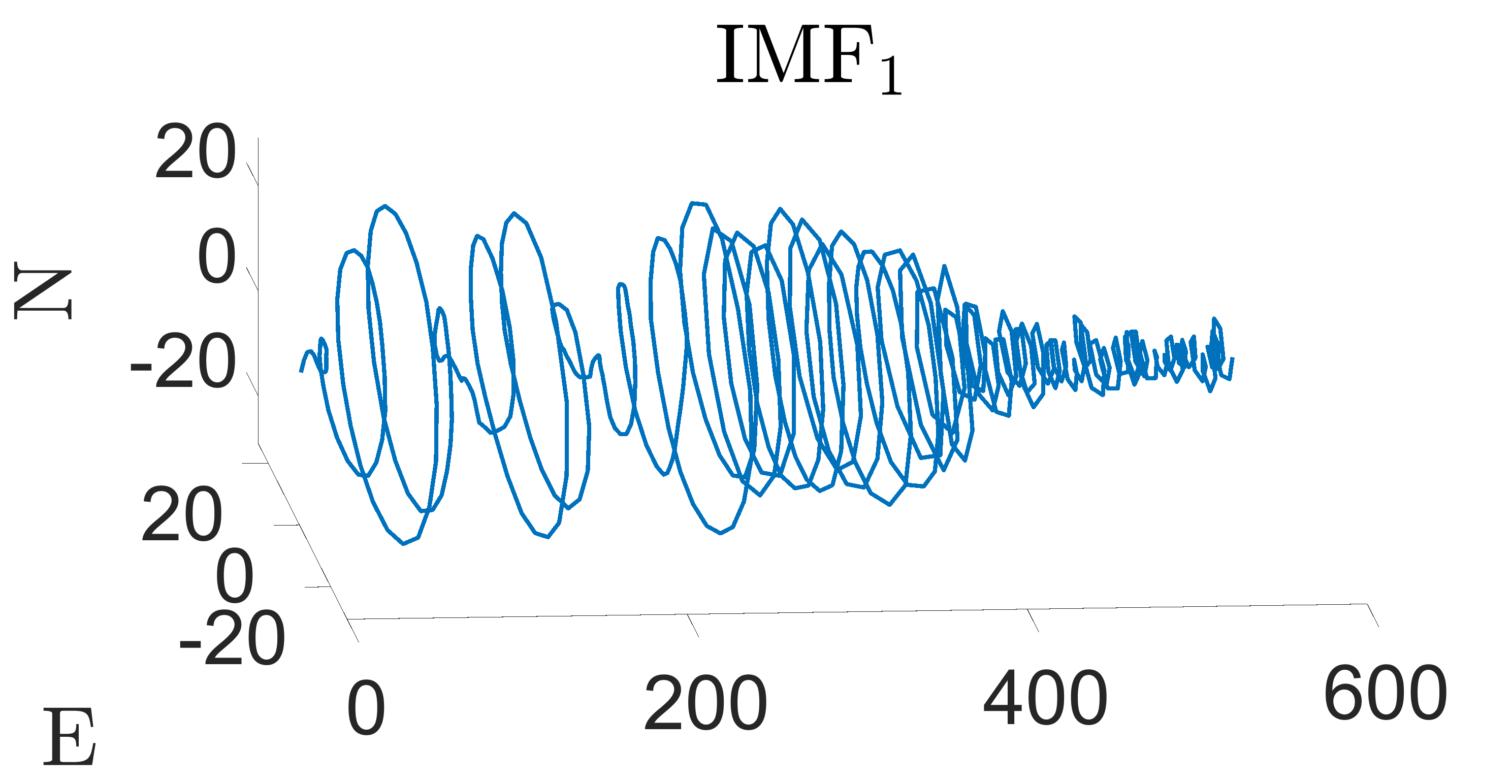} }}~
    \subfloat{{\includegraphics[width=0.24\textwidth]{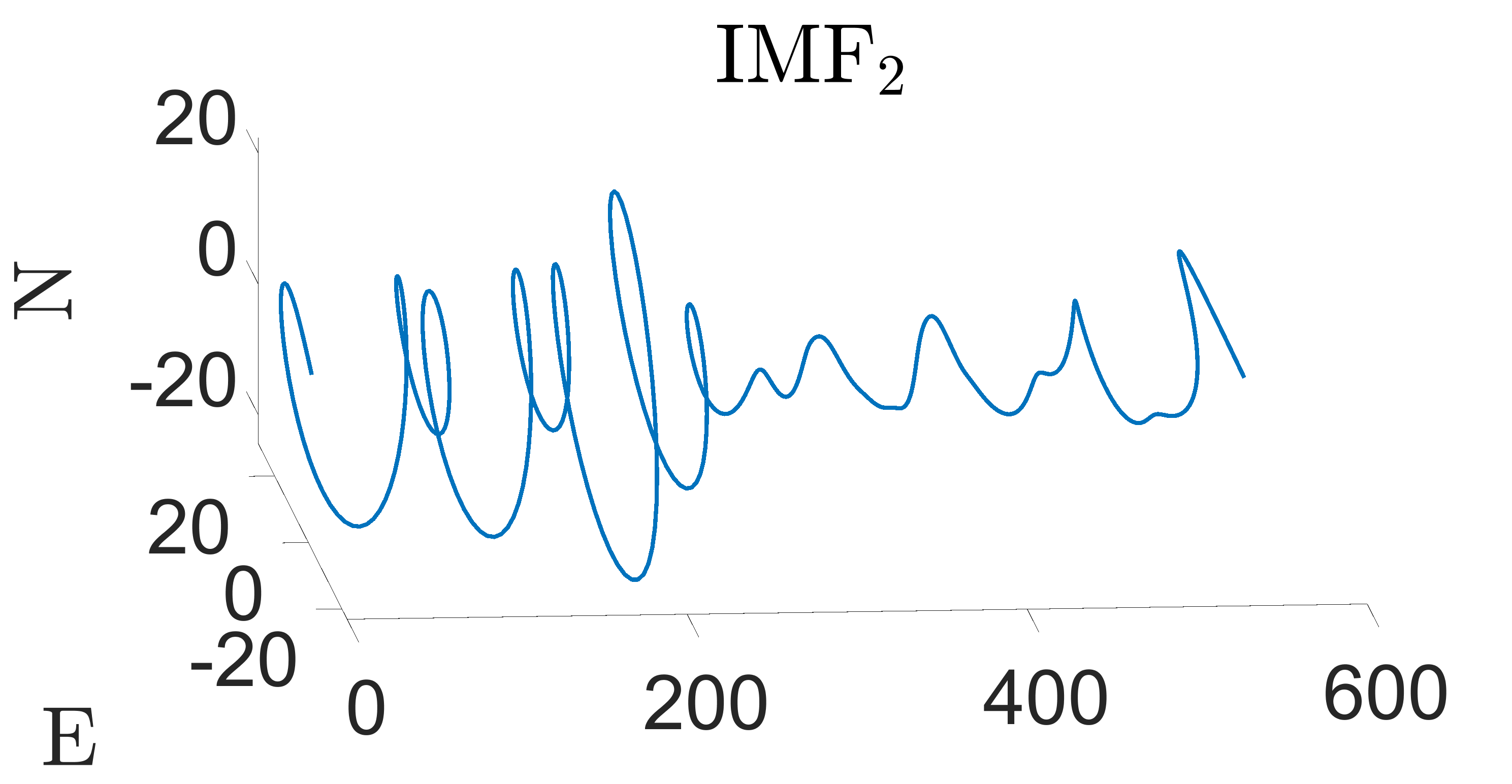} }}\\
    \subfloat{{\includegraphics[width=0.24\textwidth]{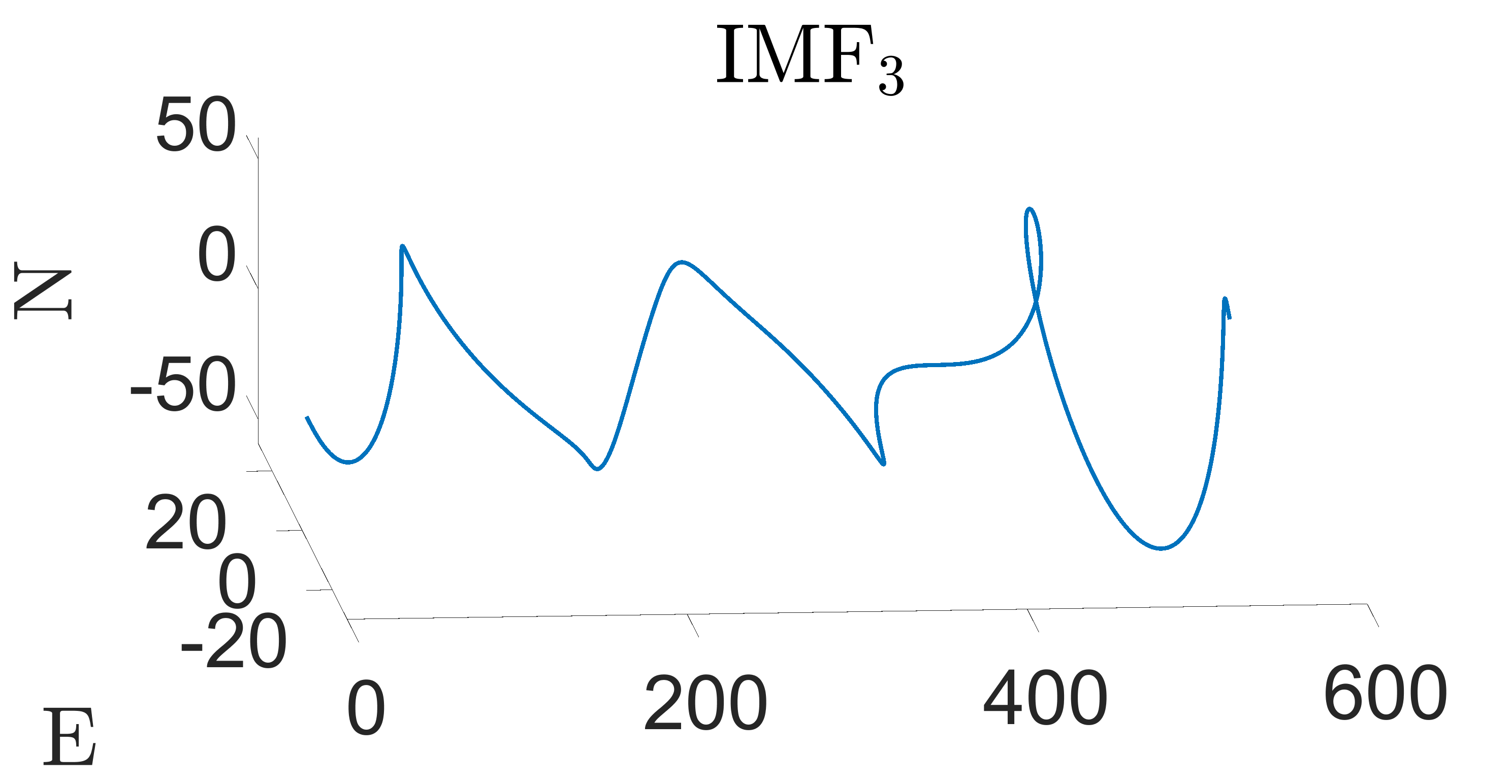} }}~
    \subfloat{{\includegraphics[width=0.24\textwidth]{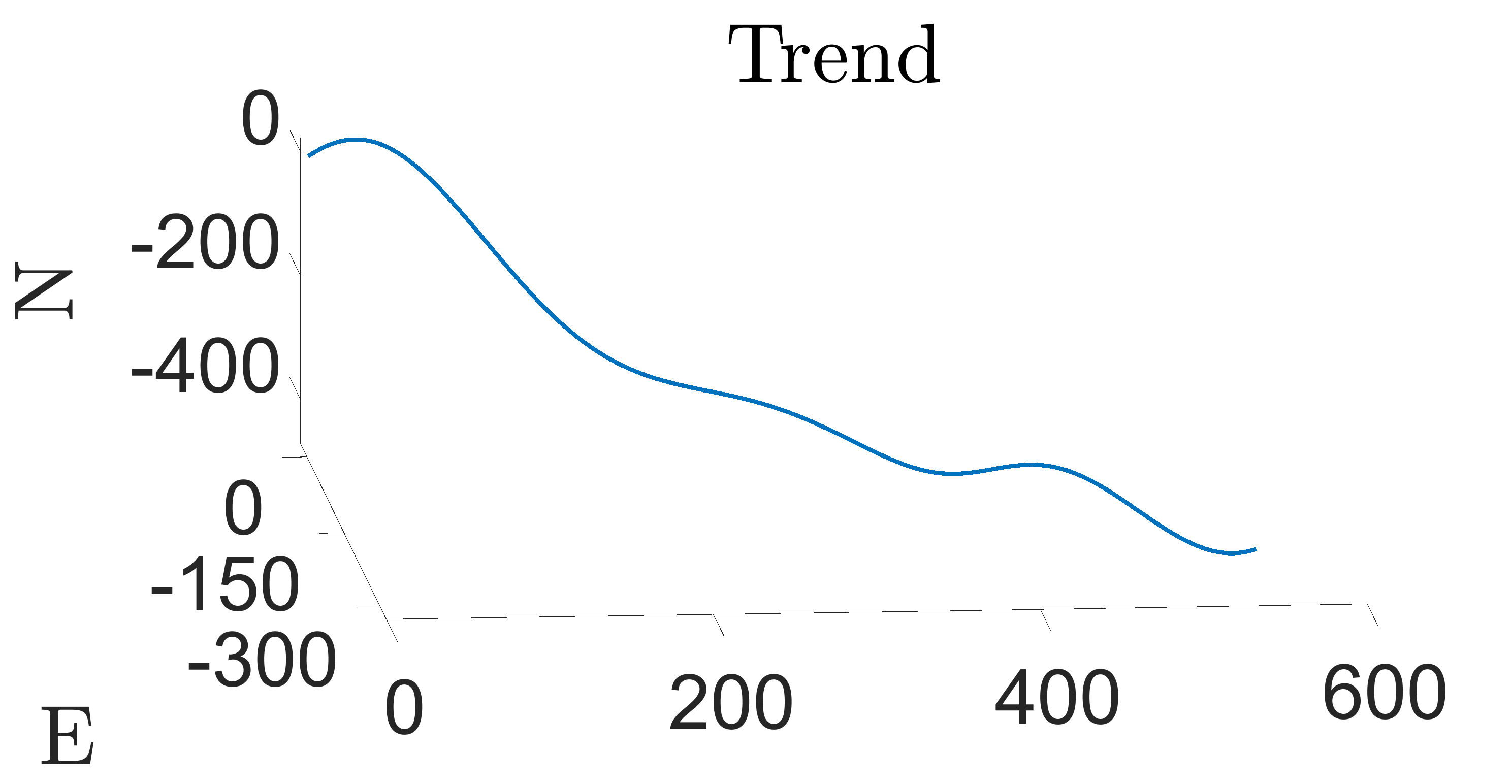} }}
    \caption{Bivariate IMFs plotted in 3D}\label{fig:Ex1_IMFs_3D}
\end{figure}

In Figure \ref{fig:Ex1_IMFs} top row we plot the two input channels. If we apply the proposed MvFIF method to this bivariate data set we produce a decomposition which contains 3 principal multivariate modulated oscillations, the so called IMFs, and a trend. This decomposition is shown as time plots in Figure \ref{fig:Ex1_IMFs} from the second row downward, and in Figure \ref{fig:Ex1_IMFs_3D} as 3D rotating modes over time. If we compare the 3D plots with the corresponding time plots of Figure \ref{fig:Ex1_IMFs}, it is evident that joint or common frequency scales across data channels explain the presence of rotations or multivariate modulated oscillations in the 3D plots of the data, which may correspond to coherent vortexes.
The non--rotating trend and sections of some IMFs are characterized by the absence of a joint frequency scale in the channels. The results produced perfectly match what was originally reported in \cite{ur2019multivariate}. This real life example shows the ability of MvFIF to properly separate multivariate oscillations and guarantee the alignment of common frequencies across multiple channels of a single IMF.

\subsection{Frequencies alignment property}\label{sec:freq-align}

In this example we further investigate the ability of MvFIF to guarantee the alignment of common oscillations across multiple channels of a single IMF. This is a fundamental feature when it comes to real life applications, where misalignment of common frequencies across multiple channels can impair a meaningful multivariate time--frequency analysis of the signal \cite{looney2009multiscale}.

We consider, in particular, a bivariate synthetic example which contains one common oscillation in both channels plus distinct nonstationary oscillations with different frequency ranges in each channel.
The first channel contains the following contributions

{\scriptsize
$$c_1(t)=\frac{1}{2}t+\left(\frac{t}{30}+\frac{3}{5}\right)\sin\left(2\pi t+\frac{\pi}{2}\right)+\left(2-\frac{t}{30}\right)\cos\left(\frac{2}{10}\pi t^{1.3}\right)$$
}
whereas the second channel is given by
{\scriptsize
$$c_2(t)=-\frac{1}{5}t+\left(2-\frac{t}{30}\right)\sin\left(2\pi t\right)+\left(\frac{1}{2}+\frac{t}{30}\right)\sin\left(6\pi \left(\frac{1}{20}t^{1.5}+t\right)\right)$$
}
The two channels are plotted as time plots in the top row of Figure \ref{fig:Ex2_IMFs_MvFIF}.
\begin{figure}%
    \centering
    \subfloat{{\includegraphics[width=0.5\textwidth]{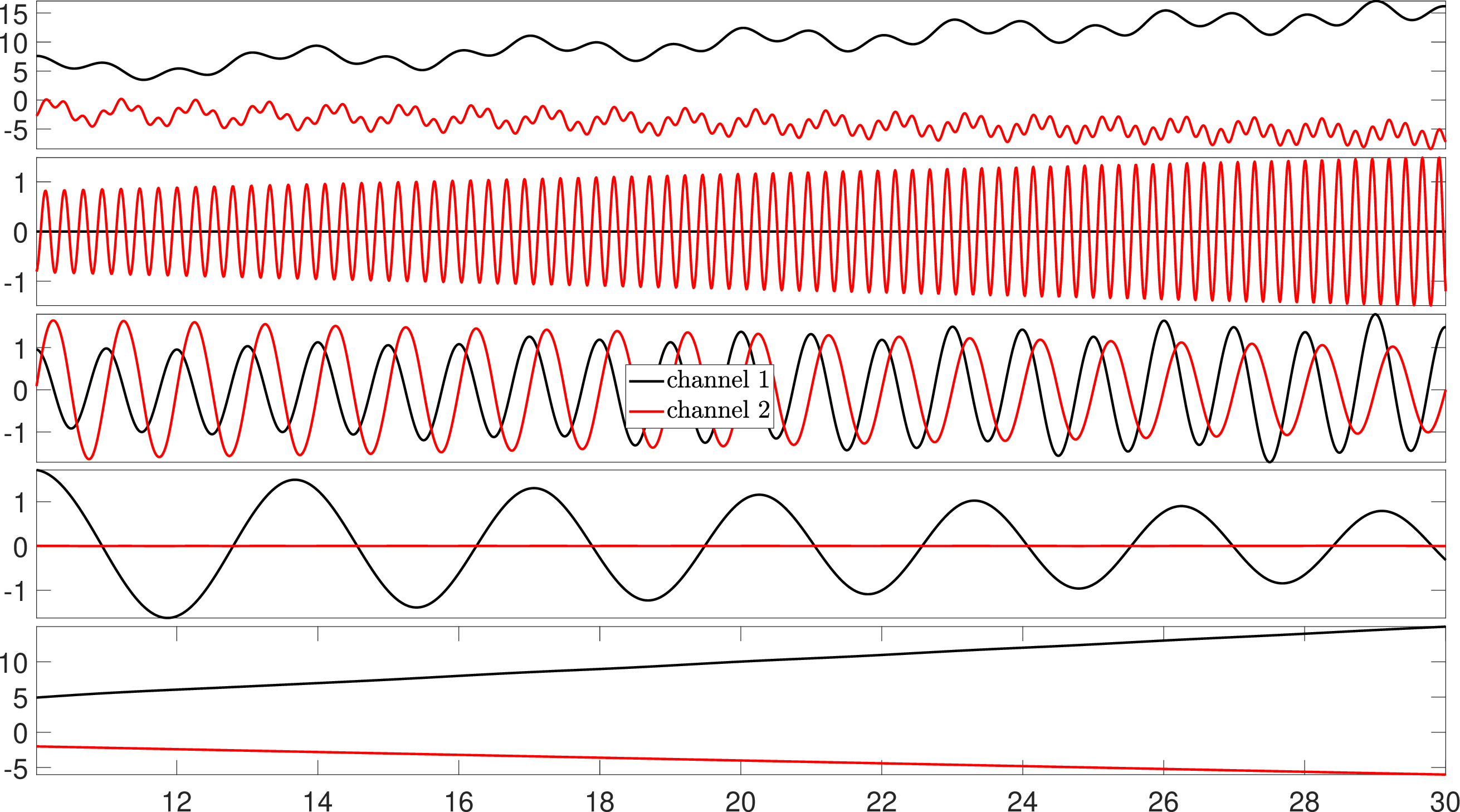} }}
    \caption{MvFIF decomposition of a bivariate artificial signal. It is evident the alignment of common oscillations across multiple channels of each IMF}\label{fig:Ex2_IMFs_MvFIF}
\end{figure}

The MvFIF decomposition is shown in Figure \ref{fig:Ex2_IMFs_MvFIF}, from the second row downward. From this result it is possible to observe that all frequencies are correctly aligned. In particular the second IMF, depicted in the third row, contains, as expected, the $2$ Hz components which are present in both channels with a phase shift of $\frac{\pi}{2}$ between the channels. The frequency alignment in the MvFIF decompositions is guaranteed by the uniform selection of the filter length among all the channels, as described in Section \ref{sec:MvFIF}.

\begin{figure}%
    \centering
    \subfloat{{\includegraphics[width=0.5\textwidth]{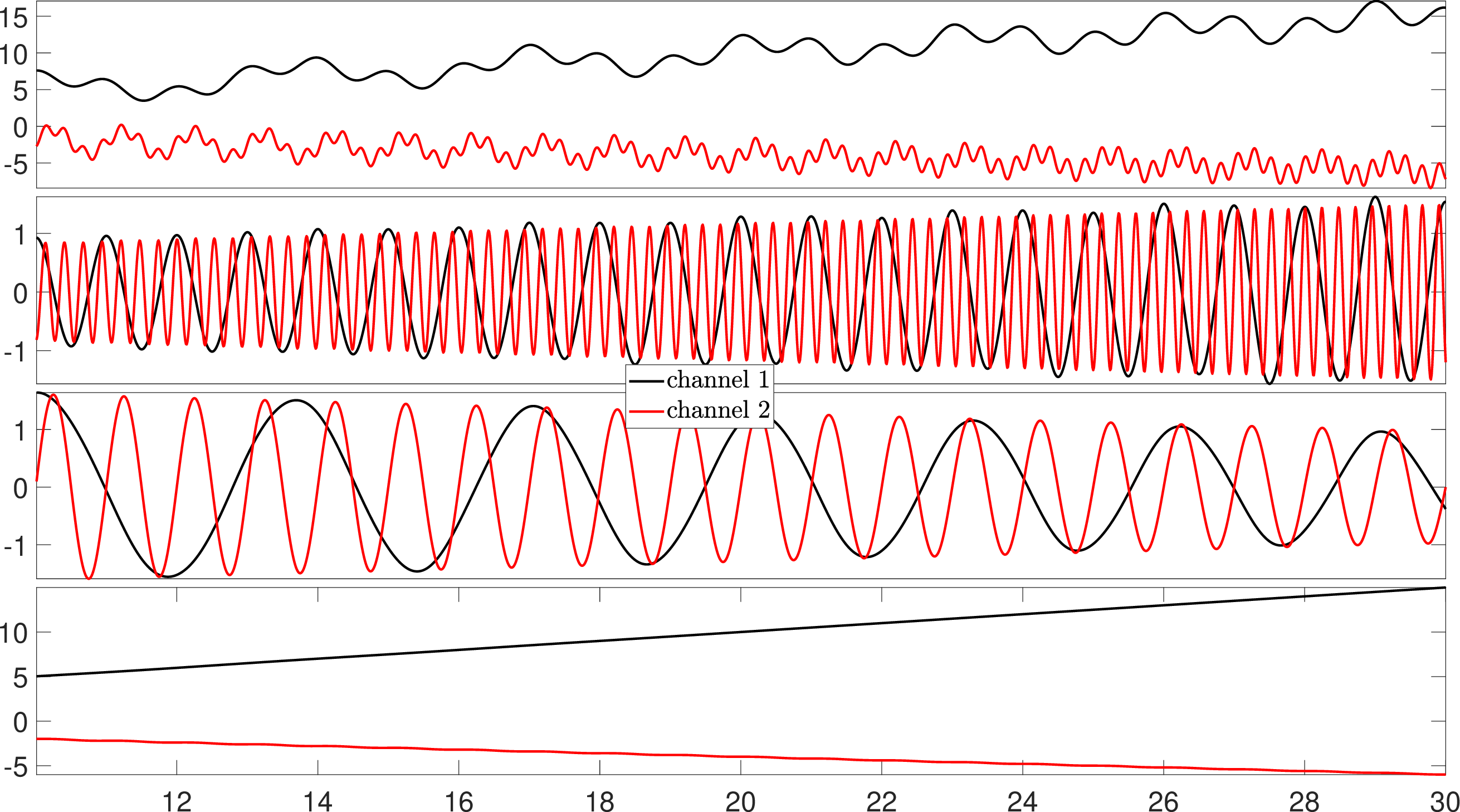} }}
    \caption{FIF decomposition of the same bivariate artificial signal with clear misalignment of the frequencies}\label{fig:Ex2_IMFs_FIF}
\end{figure}

In Figure \ref{fig:Ex2_IMFs_FIF}, we present the result of the decomposition of the same multivariate signal as obtained by means of the standard FIF method when applied to each channel separately. It is evident that the frequency alignment across channels is lost in this case. In particular we notice that the common frequency of $2$ Hz, which is present in both channels, is no more aligned. In fact it is now contained  in the first IMF for the first channel and the second IMF for the second channel.

\subsection{Filterbanks for white Gaussian noise}

The EMD and MEMD had been shown to exhibit a quasi--dyadic filterbank structure for white Gaussian noise (wGn) \cite{flandrin2004empirical}, similarly to what was done regarding the wavelet transform and its dyadic filterbank property. This is one of the key ingredients that have contributed to the success of the wavelet transform first and to the EMD after in a wide variety of applications.

In the following, we study the filterbank properties of the proposed method. We point out that, even though the Iterative Filtering method has been extensively studied in the past \cite{cicone2016adaptive,cicone2020Direct}, its filterbank properties have never been analyzed before.

\begin{figure}%
    \centering
    \subfloat{{\includegraphics[width=0.5\textwidth]{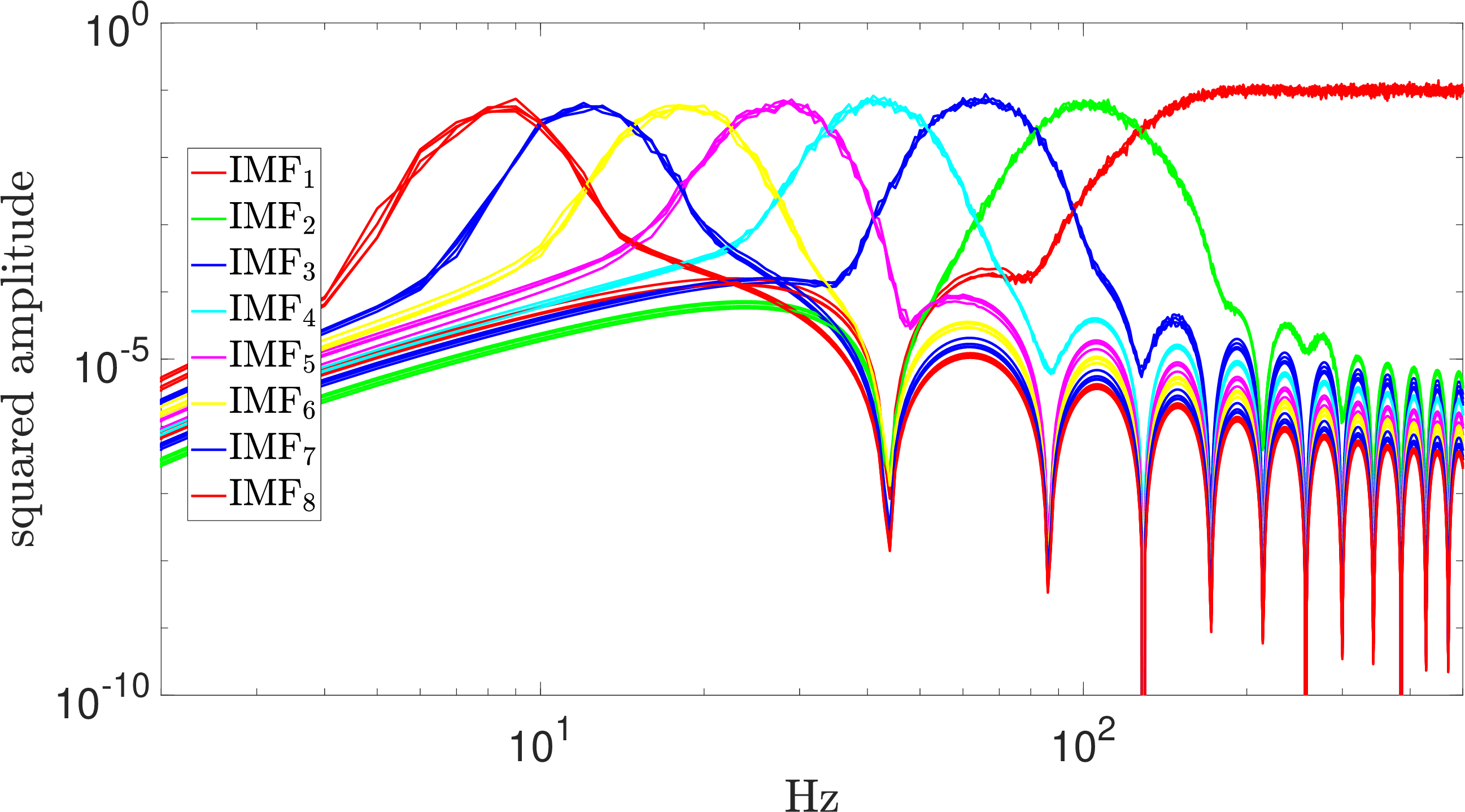} }}
    \caption{MvFIF filterbank for white Gaussian noise}\label{fig:Ex3_PSD_MvFIF}
\end{figure}

\begin{figure}%
    \centering
    \subfloat{{\includegraphics[width=0.5\textwidth]{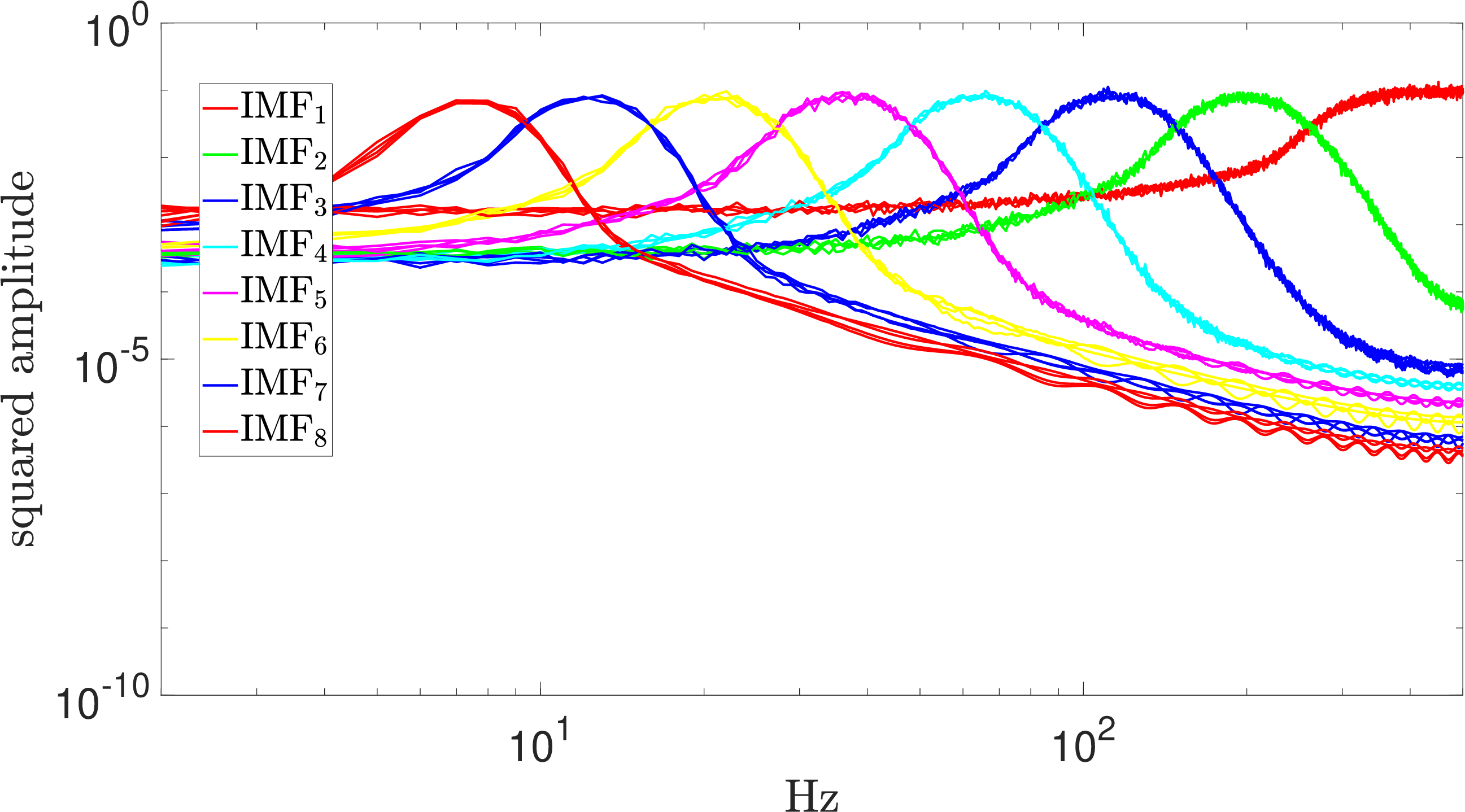} }}
    \caption{MEMD  filterbank for white Gaussian noise}\label{fig:Ex3_PSD_MEMD}
\end{figure}

We consider 100 realizations of a four--channel wGn process containing 1000 sample points each and we decompose them with MvFIF and MEMD methods. The power spectral density (PSD) plots averaged over 100 realizations are plotted in Figures \ref{fig:Ex3_PSD_MvFIF} and \ref{fig:Ex3_PSD_MEMD} for MvFIF and MEMD, respectively. By comparing the two set of plots, we can observe that both methods exhibit quasi--dyadic filterbank structure for wGn as evident by almost similar bandwidths of different IMFs in the log--frequency scale. From these figures it is also possible to notice that, as expected from the previous results, there is a good frequency alignment among different channels within corresponding IMFs.

We point out that the first IMF in the MvFIF tends to contain a wide interval of frequencies. This is probably due to the FIF mask length tuning parameter, the so called $\xi$ input parameter introduced in \eqref{eq:Unif_Mask_length}. From the comparison of Figure \ref{fig:Ex3_PSD_MvFIF} and Figure \ref{fig:Ex3_PSD_MEMD}, we observe also that the MvFIF PSD curves tend to have smaller values far away from their peaks with a peculiar shape around their local minima. The shape of the PSD curves is probably induced by the spectral properties of the filter used in the MvFIF decomposition.
The analysis of how the decomposition is influenced by the selection of the $\xi$ input parameter as well as by the choice of a filter shape has never been systematically studied in the past. This topic is out of the scope of this work, and we plan to study it in a future work.

From a computational time prospective, we report that the MvFIF required $\approx 1.57$ s to produce the multivariate decomposition of the ensemble of 4 channels wGn signals, whereas MEMD required $\approx 480.70$ s.

\subsection{Quasi--Orthogonality of MvFIF IMFs}

In standard signal processing techniques, like the Short Time Fourier Transform (STFT) and wavelet transform, the decomposition of a signal is done by means of predefined basis functions which are chosen to be orthogonal. This feature of the selected basis ensures that there is no ``leakage'' of information across different IMFs. Iterative decomposition methods like EMD and FIF are, instead, adaptive in nature. In particular, they do not require to make any assumption on the basis required to decompose a signal. Therefore we can only test for orthogonality after the decomposition has been produced.

In this section, we study the orthogonality of MvFIF IMFs using as test signal the ensemble of 100 realizations of a four-channel wGn data set which we studied previously to investigate the filterbank structure of the proposed algorithm. To measure the orthogonality of the IMFs we define the following \emph{correlation coefficient}
$$c_{ij}=\frac{\textrm{cov}(\IMF_i,\ \IMF_j)}{\sigma_i\sigma_j} $$
where $\textrm{cov}(\cdot)$ represents the covariance between two signals and $\sigma_i$ is the standard deviation of the i--th IMF. If $c_{ij}$ is close to zero the i--th and j--th IMFs can be considered to be quasi-orthogonal. If, instead, this value is close to one, the two IMFs are strongly correlated.

\begin{figure}%
    \centering
    \subfloat{{\includegraphics[width=0.245\textwidth]{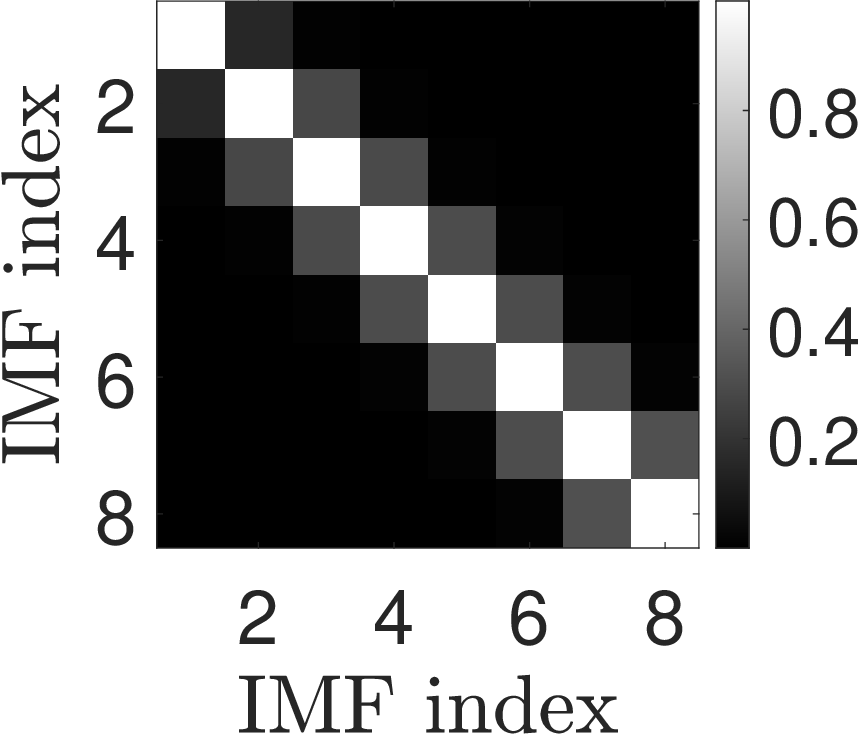} }}~\subfloat{{\includegraphics[width=0.245\textwidth]{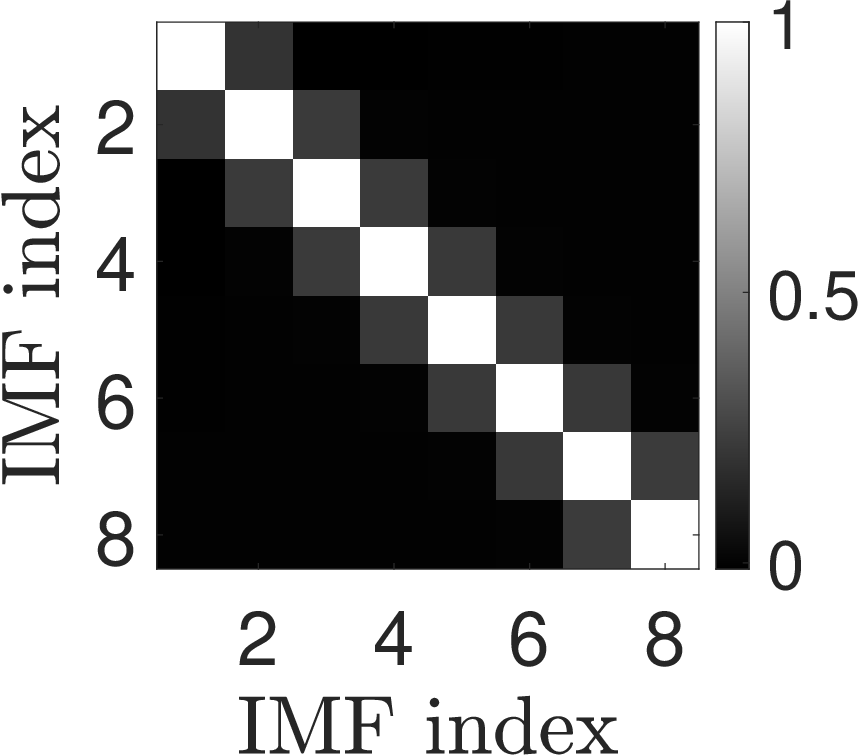} }}
    \caption{MvFIF and MEMD correlation coefficient matrices, respectively on the left and right panel}\label{fig:Ex4_corr}
\end{figure}

In Figure \ref{fig:Ex4_corr} we show the correlation coefficient matrices obtained from MvFIF, left panel, and MEMD, right, when applied to the 4--channel wGn ensemble, which we averaged over all the 100 realizations and channels. We notice the almost diagonal structure of both correlation matrix which suggests a strong quasi--orthogonality of the IMFs produced by both techniques.

\subsection{Noise Robustness}

In this section we show the robustness of MvFIF decomposition in presence of noise and compare performance with other methods proposed so far in the literature. In the following, we consider the signal to noise ratio (SNR) measured in dB
$$\textrm{SNR}=20\log(\|\textrm{signal}\|_2/\|\textrm{noise}\|_2)$$
The IF method, and its fast implementation FIF, have been already proved to be robust against noise in \cite{cicone2016adaptive,cicone2020Direct}. In order to test also MvFIF robustness, we consider the following bivariate signal which we perturb with additive Gaussian noise at different noise levels. 
$$c_1(t)=\frac{1}{2}t+\sin\left(2\pi t+\frac{\pi}{2}\right)+\cos\left(\frac{2}{10}\pi t^{1.3}\right)$$
where $c_{11}(t)$, $c_{12}(t)$, and $c_{13}(t)$ correspond to the three components, respectively.
$$c_2(t)=-\frac{1}{5}t+\sin\left(2\pi t\right)+\sin\left(6\pi \left(\frac{1}{20}t^{1.5}+t\right)\right)$$
whose components are labeled $c_{21}(t)$, $c_{22}(t)$, and $c_{23}(t)$, respectively. 
We test the performance of MvFIF at four different couples of SNR levels:  (15,2), (9,8),  (23,16), (35,22), where the first entry correspond to the first channel. For each SNR couple we produce one hundred different realizations of Gaussian noise that we add to the two channels. 


In figure \ref{fig:Ex5_MvFIF}, top rows, we show an example of the previously defined bivariate signal perturbed by additive noise, when the SNR is (15,2) dB. In the same figure we report the MvFIF decomposition. The IMFs of the two channels, in black, are plotted against the unperturbed signal, in red.

\begin{figure}%
    \centering
    \subfloat{{\includegraphics[width=0.5\textwidth]{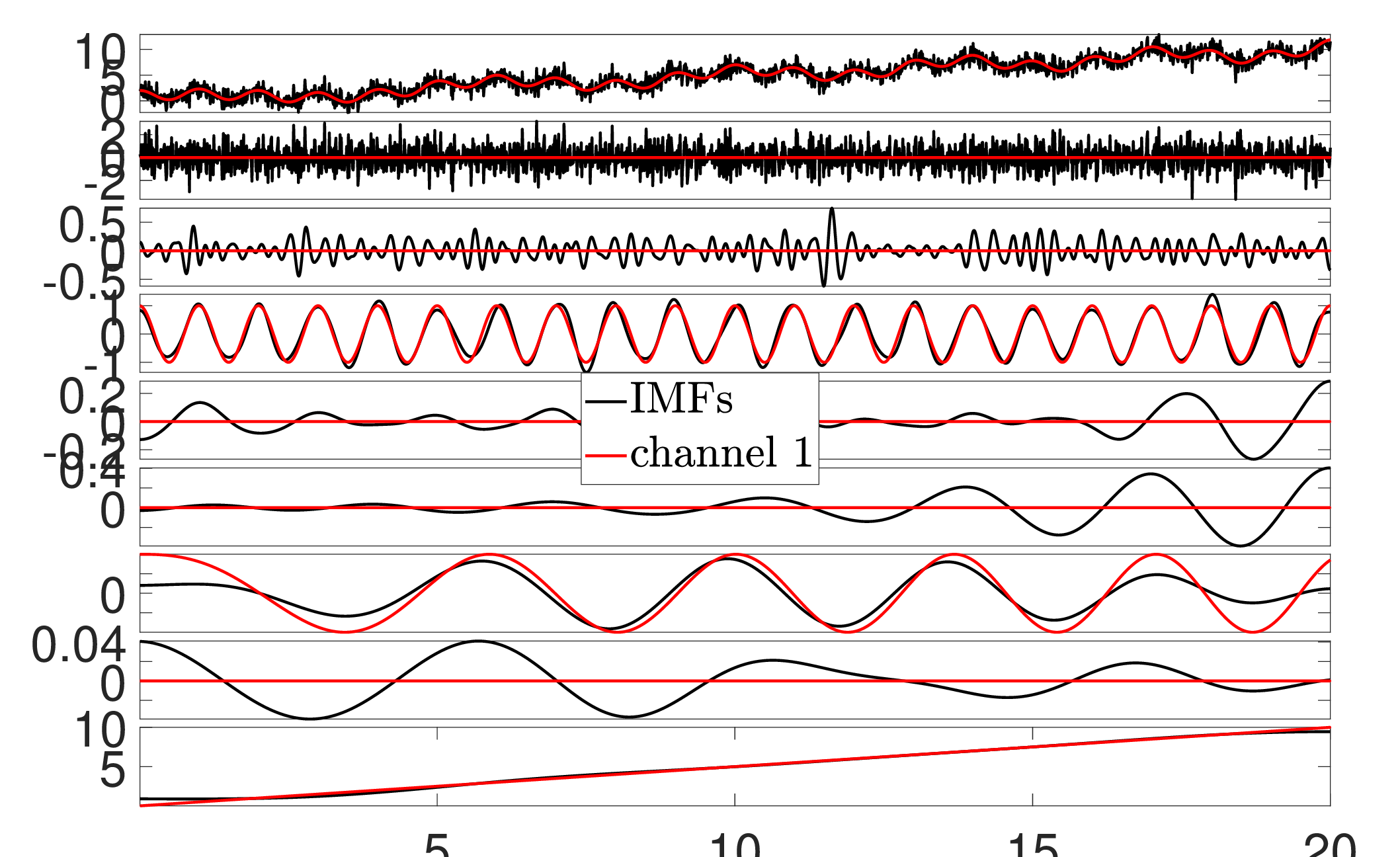} }}\\
    \subfloat{{\includegraphics[width=0.5\textwidth]{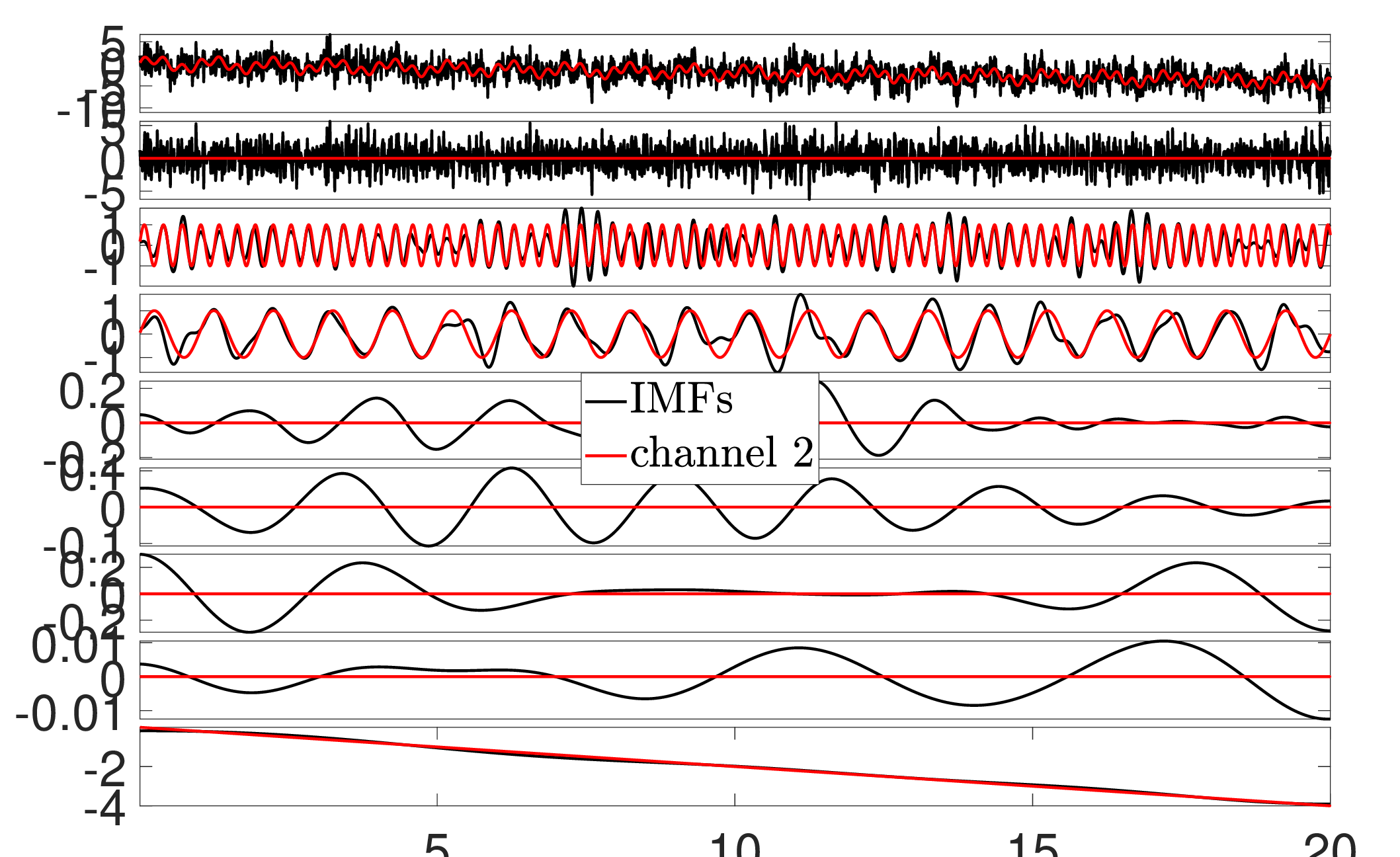} }}
    \caption{Example of robustness against noise of a MvFIF decomposition. The bivariate signal considered in this case, which is depicted in the top rows of the top and bottom panels, has a SNR (15,2). The IMFs produced by MvFIF are reported in rows from 2 to 8.}\label{fig:Ex5_MvFIF}
\end{figure}

We observe that, in the example depicted, the method is perfectly able to reconstruct the two oscillatory components of each clean channel and their trends and separate them from the noise contributions. In particular, in the top panel of Figure \ref{fig:Ex5_MvFIF}, we see that IMF 3 and 6 contain the oscillatory components of the first channel, i.e. $c_{12}$ and $c_{13}$, whereas in the bottom panel, we recognize the oscillatory components of the second channel, i.e. $c_{22}$ and $c_{23}$, as IMF 2 and 3. All these IMFs are, as expected, slightly perturbed given the SNR level of this signal. Already from this simple result it becomes clear the robustness of the method, even in presence of heavy noise.


\begin{figure}%
	\centering
	\subfloat{{\includegraphics[width=0.5\textwidth]{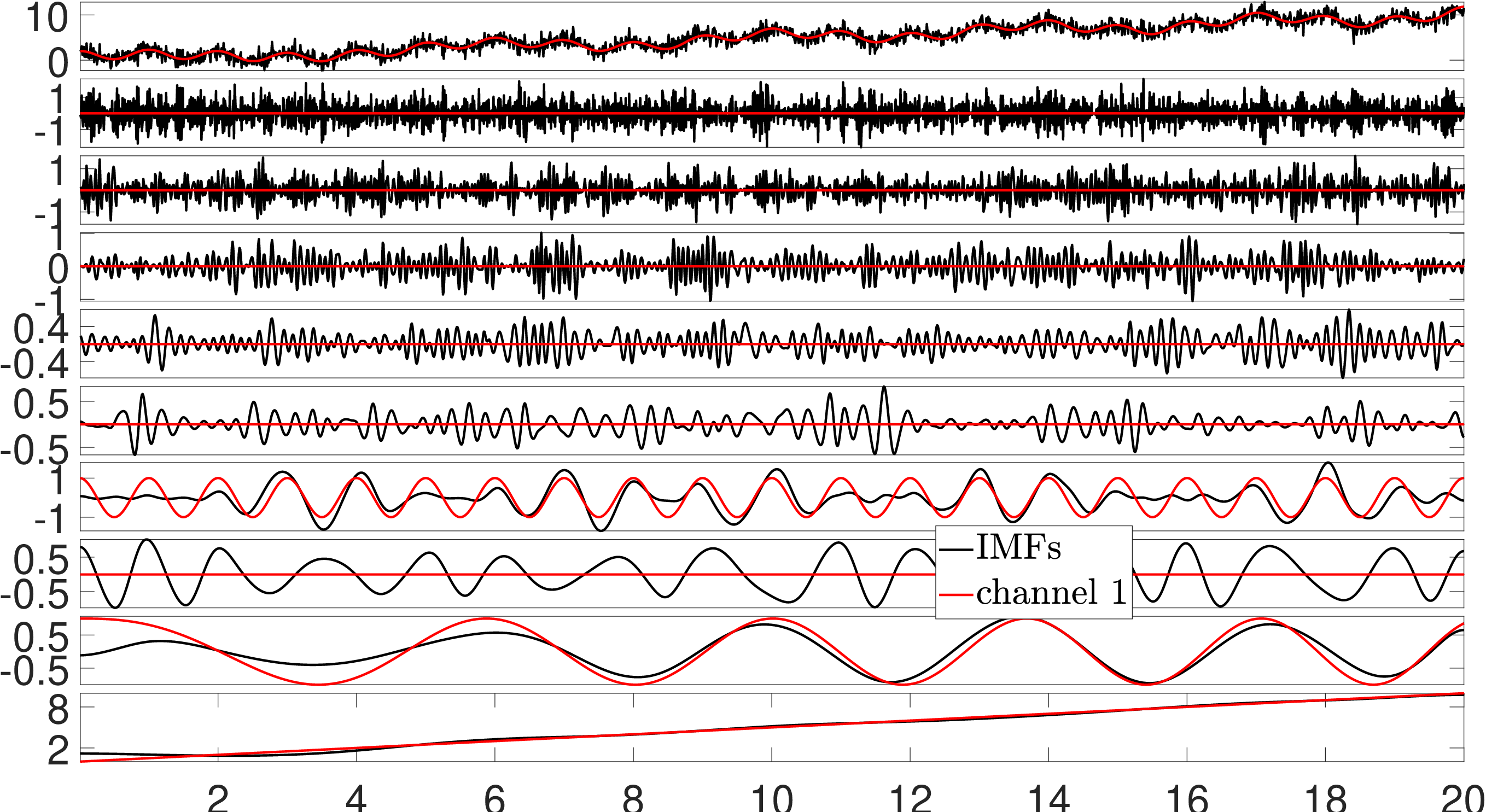} }}\\
	\subfloat{{\includegraphics[width=0.5\textwidth]{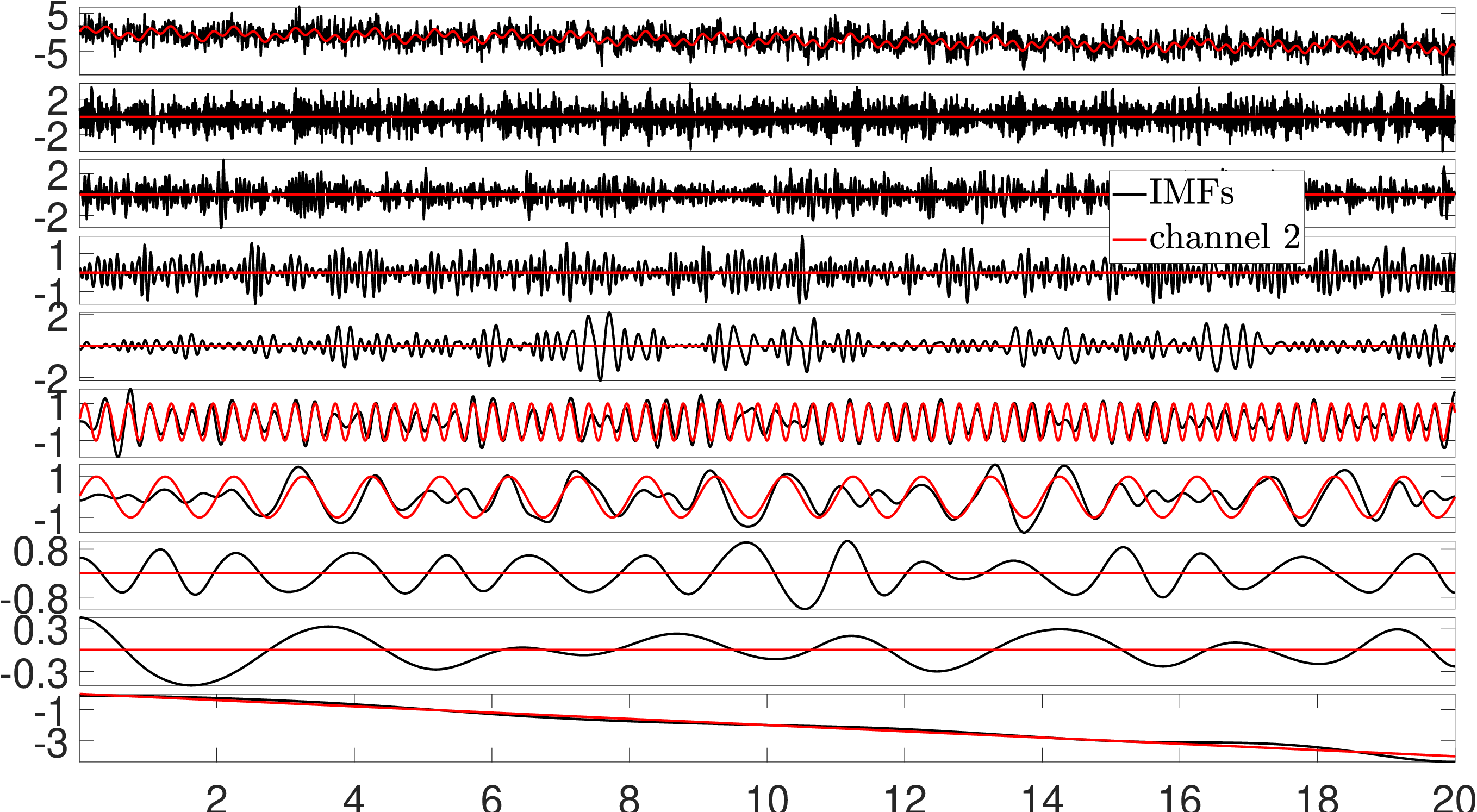} }}
	\caption{MEMD decomposition of the sample bivariate perturbed signal with SNR couple (15,2).}\label{figS:Ex5_MEMD}
\end{figure}

In Figure \ref{figS:Ex5_MEMD} we report the MEMD decomposition of this example to compare it with the results obtained by MvFIF applied on the same example. Like for the MvFIF decomposition, also MEMD method is able to separate the meaningful IMFs from the noise contributions (IMF 6 and 8 for the first channel, and IMF 5 and 6 for the second). However it is possible to observe a lower quality in the reconstruction. As known, in fact, EMD-based methods are sensitive to noise. To reduce the impact of the noise it is possible to increase the number of reprojections of the signals to a value higher than the default 64, as explained in \cite{rehman2009Multivariate}, at the price of increasing the computational cost of the method.


To provide a quantitative indicator of the quality and performance of MvFIF compared to MEMD, mentioned earlier, FA-MVEMD\footnote{The FA-MVEMD Matlab version is available at \url{https://it.mathworks.com/matlabcentral/fileexchange/71270-fast-and-adaptive-multivariate-and-multidimensional-emd}} \cite{thirumalaisamy2018fast}, the MVMD\footnote{The MVMD Matlab version is available at \url{https://it.mathworks.com/matlabcentral/fileexchange/72814-multivariate-variational-mode-decomposition-mvmd}} \cite{ur2019multivariate}, and the MSSA \cite{golyandina2015MSSA} methods, we fix the SNR of the two channels to be, in dB, either the couple (15,2), or (9,8), or (23,16), or (35,22), and then we consider one hundred different realizations of Gaussian noise added to the two channels. For each realization of noise we study the performance of all the aforementioned techniques using the Reconstructed Quality Factor (RQF) \cite{fourer2016recursive} which is defined as follows:
\begin{equation}
	RQF_{\overline{\textbf{v}}_k-\textbf{v}_k}=20\log_{10} \frac{\left\|\textbf{v}_k\right\|_2}{\left\| \overline{\textbf{v}}_k-\textbf{v}_k \right\|_2}
\end{equation}
\noindent where $\textbf{v}_k=\left[v_{k}(i)\right]_{i=1,\ldots, n}$, and $\overline{\textbf{v}}_k=\left[v_{k}(i)\right]_{i=1,\ldots, n}$ are the actual and reconstructed \emph{k}-th component of the signal. We point out that some authors refer to this factor as Quality Reconstruction Factor (QRF), see, for instance, \cite{jain2020}.

To summarize the results, we report in a table the mean value and, in between parentheses, the standard deviation of the RQF per channel and per couple of SNR.

In Tables \ref{tabS:Noise_1} we report the results for the SNR couple (15,2) dB,

\begin{center}
	\begin{tabular}{c|c|c|c|c|c|c}
		
		& \multicolumn{3}{c|}{Channel 1 } & \multicolumn{3}{c}{Channel 2 }\\
		\cline{2-7}
		
		& $c_{13}$ &   $c_{12}$  &   $c_{11}$ &  $c_{23}$ &   $c_{22}$  &   $c_{21}$\\
		\hline
		MvFIF     &  28.09 (2.59) &  \textbf{13.45 (1.47)}  &  6.95 (1.72)&   \textbf{26.13 (2.63)}  &  \textbf{7.18 (1.67)} &   \textbf{4.10 (0.87)}\\ 
		FA-MVEMD  &   \textbf{29.20 (2.25)} &  10.41 (1.75)  &  \textbf{8.04 (1.93)} &     24.96 (2.66)  &  4.48 (2.14) &   3.50 (1.03)\\ 
		MVMD      &  16.28 (0.18) &   0.36 (0.09)  & -0.43 (0.05) &   11.55 (0.46)  & -3.44 (0.27) &   3.62 (0.51) \\ 
		MEMD      &  8.01 (1.38) &   4.03 (2.41)  &  5.51 (2.78)&  8.07 (1.56)  &  2.59 (1.74) &   3.26 (0.78)\\ 
		MSSA      &   -0.02 (0.00) &  -7.16 (8.57)  & -7.16 (8.59)&   -0.00 (0.00)  & -5.87 (3.77) &  -5.98 (3.74) \\ 
	\end{tabular}
	\captionof{table}{RQF mean value and standard deviation, in between parentheses, of the two channels when the SNR is of $15$ dB for the first channel, and $2$ dB for the second one.}\label{tabS:Noise_1}	
\end{center}

Average computational times and the corresponding standard deviations for the different techniques compared in this example are listed in Table  \ref{tabS:Noise_time_1}.

\begin{center}
	\begin{tabular}{c||c|c|c|c|c}
		& MvFIF    &   FA-MVEMD & MVMD   & MEMD     & MSSA \\ 
		\hline
		\hline
		Time (s) & \textbf{ 0.25 (0.07)} &   0.59 (0.15)  & 5.11  (2.00)   & 40.60 (4.50)  &0.49 (0.04)  \\
		\end{tabular}
		\captionof{table}{Mean computational times, in seconds, and standard deviation, in between parentheses, for different methods when the bivariate signal has SNR around $15$ dB for the first channel, and $2$ dB for the second one.}\label{tabS:Noise_time_1}
	\end{center}

In Table \ref{tabS:Noise_2} we report the results for the SNR couple (9,8) dB. The corresponding computational times are listed in Table \ref{tabS:Noise_time_2}.

\begin{center}
	\begin{tabular}{c|c|c|c|c|c|c}
		
		& \multicolumn{3}{c|}{Channel 1 } & \multicolumn{3}{c}{Channel 2 }\\
		\cline{2-7}
		
		& $c_{13}$ &   $c_{12}$  &   $c_{11}$ &  $c_{23}$ &   $c_{22}$  &   $c_{21}$\\
		\hline
		MvFIF     &   27.33 (2.29)  &    7.07 (2.25)  &    5.58 (1.85)&   29.78 (2.25)  &   11.43 (1.15)  &    9.58 (0.74)\\ 
		FA-MVEMD  &   28.42 (2.39)  &    7.16 (1.76)  &    5.69 (1.94)&   28.89 (2.81)  &   10.26 (1.36)  &    5.96 (1.75)\\ 
		MVMD      &   15.90 (0.36)  &   -0.91 (0.21)  &   -1.51 (0.17)&   12.43 (0.27)  &   -2.75 (0.16)  &    9.16 (0.46)\\ 
		MEMD      &    7.94 (0.98)  &    2.83 (1.61)  &    5.30 (2.38)&    8.03 (1.05)  &    4.02 (2.12)  &    5.74 (1.24)\\ 
		MSSA      &   -0.00 (0.00)  &   -1.64 (4.96)  &   -1.64 (4.96)&   -0.10 (0.02)  &   -2.03 (2.55)  &   -2.27 (2.49)\\ 
	\end{tabular}
	\captionof{table}{RQF mean value and standard deviation, in between parentheses, of the two channels when the SNR is of $9$ dB for the first channel, and $8$ dB for the second one.}\label{tabS:Noise_2}
\end{center}

\begin{center}
	\begin{tabular}{c||c|c|c|c|c}
		& MvFIF    &   FA-MVEMD & MVMD   & MEMD     & MSSA \\ 
		\hline
		\hline
		Time (s) & 0.08 (0.06)  &  0.37 (0.12)  & 3.52 (1.41)   & 39.75 (5.02)  & 0.33 (0.02)  \\
	\end{tabular}
	\captionof{table}{Average computational time and, in between parentheses, the standard deviation, in seconds, for different methods when the SNR couple is (9,8).}\label{tabS:Noise_time_2}
\end{center}

Table \ref{tabS:Noise_3} contains the results for the SNR couple (23,16) dB. The corresponding computational times are listed in Table \ref{tabS:Noise_time_3}.

\begin{center}
	\begin{tabular}{c|c|c|c|c|c|c}
		
		& \multicolumn{3}{c|}{Channel 1 } & \multicolumn{3}{c}{Channel 2 }\\
		\cline{2-7}
		
		& $c_{13}$ &   $c_{12}$  &   $c_{11}$ &  $c_{23}$ &   $c_{22}$  &   $c_{21}$\\
		\hline
		MvFIF     &   27.47 (2.02)  &   19.44 (1.44)  &    7.22 (1.42)&   32.64 (1.24)  &   15.21 (1.73)  &   15.49 (0.74)\\ 
		FA-MVEMD  &   30.41 (2.80)  &   15.73 (1.68)  &   11.41 (1.89)&   28.00 (2.18)  &    7.17 (4.63)  &   10.32 (1.41)\\ 
		MVMD      &   16.45 (0.08)  &    0.90 (0.02)  &   -0.07 (0.01)&   12.83 (0.10)  &   -2.49 (0.07)  &   14.48 (0.29)\\ 
		MEMD      &    7.95 (0.63)  &   14.16 (3.35)  &    5.52 (2.49)&    7.99 (0.72)  &   12.29 (2.88)  &    6.39 (2.61)\\ 
		MSSA      &   -0.00 (0.00)  &  -10.57 (9.10)  &  -10.59 (9.11)&   -0.01 (0.01)  &   -6.10 (5.35)  &   -4.74 (5.20)\\ 
	\end{tabular}
	\captionof{table}{RQF mean value and standard deviation, in between parentheses, of the two channels when the SNR is of $23$ dB for the first channel, and $16$ dB for the second one.}\label{tabS:Noise_3}
\end{center}

\begin{center}
	\begin{tabular}{c||c|c|c|c|c}
		& MvFIF    &   FA-MVEMD & MVMD   & MEMD     & MSSA \\ 
		\hline
		\hline
		Time (s) &    0.09 (0.03) &0.36 (0.14)    &     3.71 (1.56) &   38.01 (4.71) &   0.37 (0.06)\\
	\end{tabular}
	\captionof{table}{Average computational time and, in between parentheses, the standard deviation, in seconds, for different methods when the SNR couple is (23,16).}\label{tabS:Noise_time_3}
\end{center}

The last case results, SNR couple (35,22) dB, are reported in Table \ref{tabS:Noise_4}, whereas, the corresponding computational times are listed in Table \ref{tabS:Noise_time_4}.

\begin{center}
	\begin{tabular}{c|c|c|c|c|c|c}
		
		& \multicolumn{3}{c|}{Channel 1 } & \multicolumn{3}{c}{Channel 2 }\\
		\cline{2-7}
		& $c_{13}$ &   $c_{12}$  &   $c_{11}$ &  $c_{23}$ &   $c_{22}$  &   $c_{21}$\\
		\hline
		MvFIF     &   26.00 (1.45)  &   23.49 (3.45)  &    6.79 (0.66)&   33.13 (0.73)  &   16.74 (1.87)  &   19.05 (0.46)\\ 
		FA-MVEMD  &   31.24 (2.33)  &   20.65 (2.72)  &   12.06 (2.17)&   28.88 (3.05)  &   14.96 (3.33)  &   13.62 (1.66)\\ 
		MVMD      &   19.34 (0.02)  &   17.50 (0.06)  &   -2.55 (0.01)&   30.13 (0.42)  &   18.00 (0.21)  &   18.60 (0.26)\\ 
		MEMD      &    8.04 (0.42)  &   17.44 (3.89)  &    6.31 (2.10)&    8.12 (0.49)  &   14.42 (2.93)  &   15.58 (0.81)\\ 
		MSSA      &   -0.00 (0.00)  &  -17.54 (3.61)  &  -17.56 (3.62)&   -0.00 (0.00)  &  -10.30 (1.88)  &   -8.04 (4.62)\\ 
	\end{tabular}
	\captionof{table}{RQF mean value and standard deviation, in between parentheses, of the two channels when the SNR is of $35$ dB for the first channel, and $22$ dB for the second one.}\label{tabS:Noise_4}
\end{center}

\begin{center}
	\begin{tabular}{c||c|c|c|c|c}
		& MvFIF    &   FA-MVEMD & MVMD   & MEMD     & MSSA \\ 
		\hline
		\hline
		Time & 0.13 (0.03)  &   0.41 (0.12) &   3.79 (1.89) &  42.48 (7.97) &   0.49 (0.04)\\
	\end{tabular}
	\captionof{table}{Average computational time and, in between parentheses, the standard deviation, in seconds, for different methods when the SNR couple is (35,22).}\label{tabS:Noise_time_4}
\end{center}

From the comparisons reported above it is possible to observe that MvFIF performs better than other techniques and its computational complexity is the smallest among all. As we will see in the following examples, the computational efficiency of the proposed method becomes even more evident as the number of channels in the signal increases.

\subsection{Performance evaluation for an increasing number of channels}

In this section we test the performance of MvFIF as the number of channels in a signal varies from 2 to 64. We consider an ensemble of 200 realizations of wGn processes containing 1000 sample points each, and a number of channels equal to 2, 16, 64, respectively.
\begin{center}
\begin{tabular}{c||c|c|c}
  $\#$ of Channels  & 2 & 16 & 64 \\
  \hline
  \hline
  Computational Time (s) & $2.85$ & $7.91$ & $22.49$\\
\end{tabular}
\captionof{table}{Summary of the computational times versus the number of channels}\label{tab:Ex6}
\end{center}

We decompose these signals into 8 IMFs plus a trend. The computational time to produce each decomposition is reported in Table \ref{tab:Ex6}.
\begin{figure}%
	\centering
	\subfloat{{\includegraphics[width=0.32\textwidth]{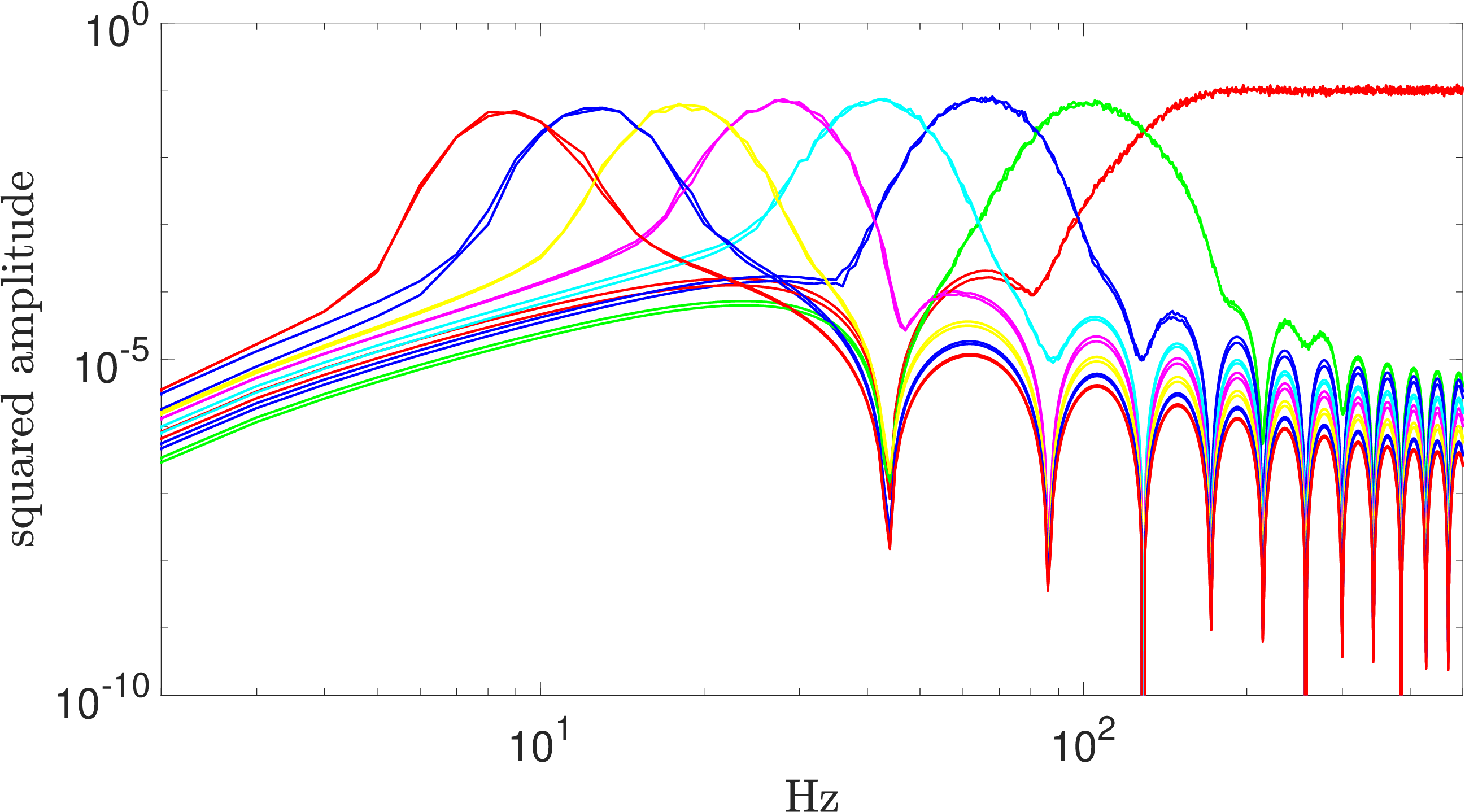} }}~\subfloat{{\includegraphics[width=0.32\textwidth]{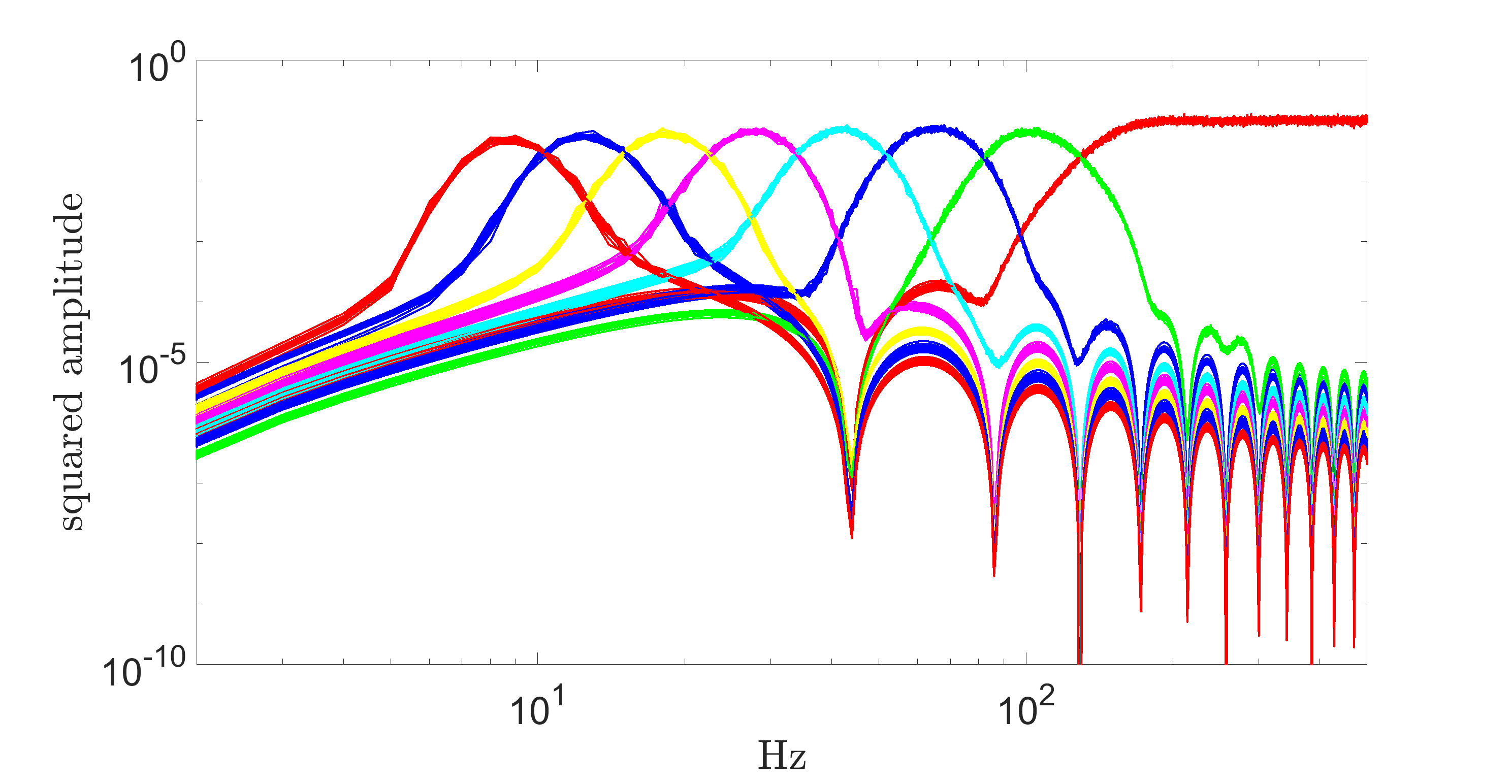} }}~\subfloat{{\includegraphics[width=0.32\textwidth]{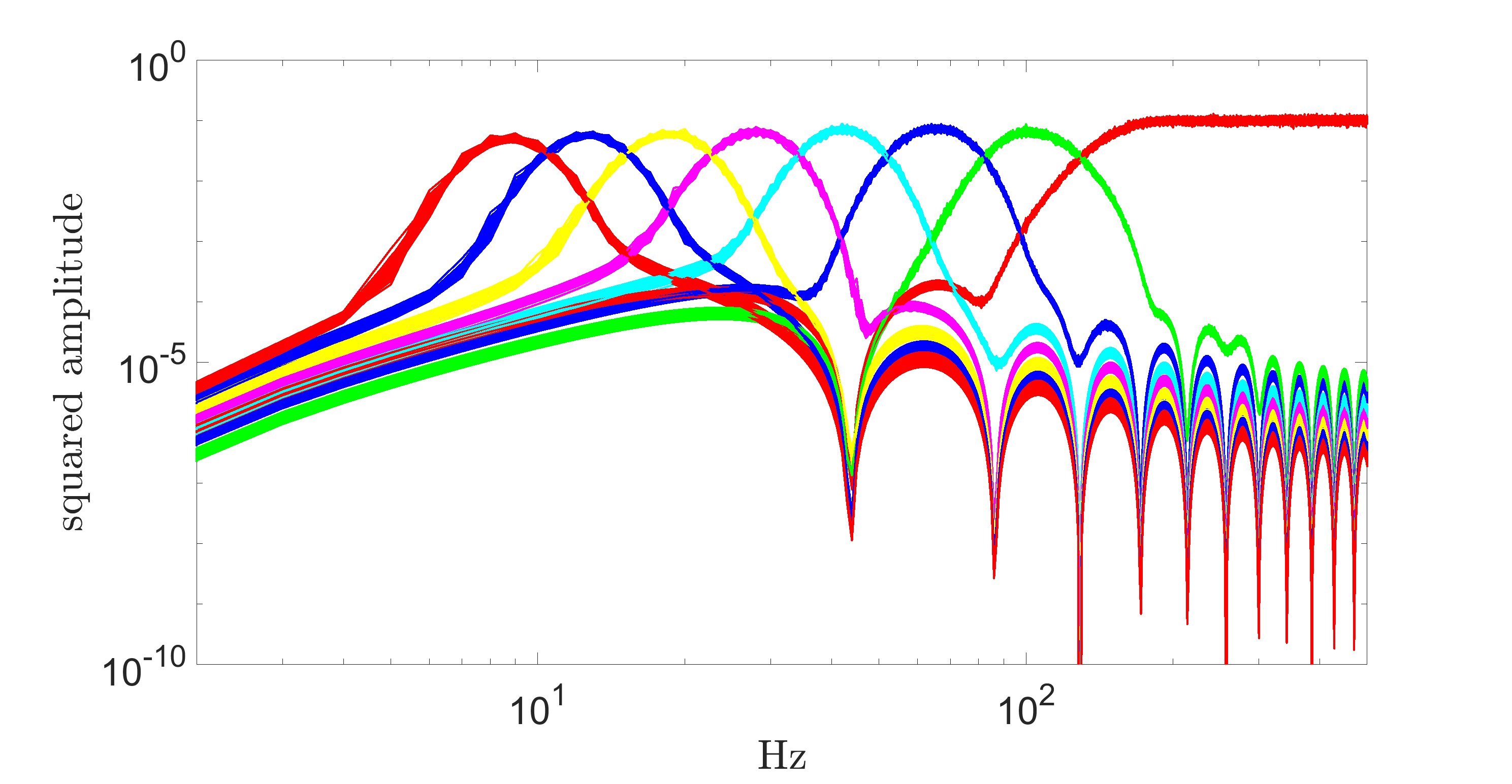} }}\\
	\subfloat{{\includegraphics[width=0.32\textwidth]{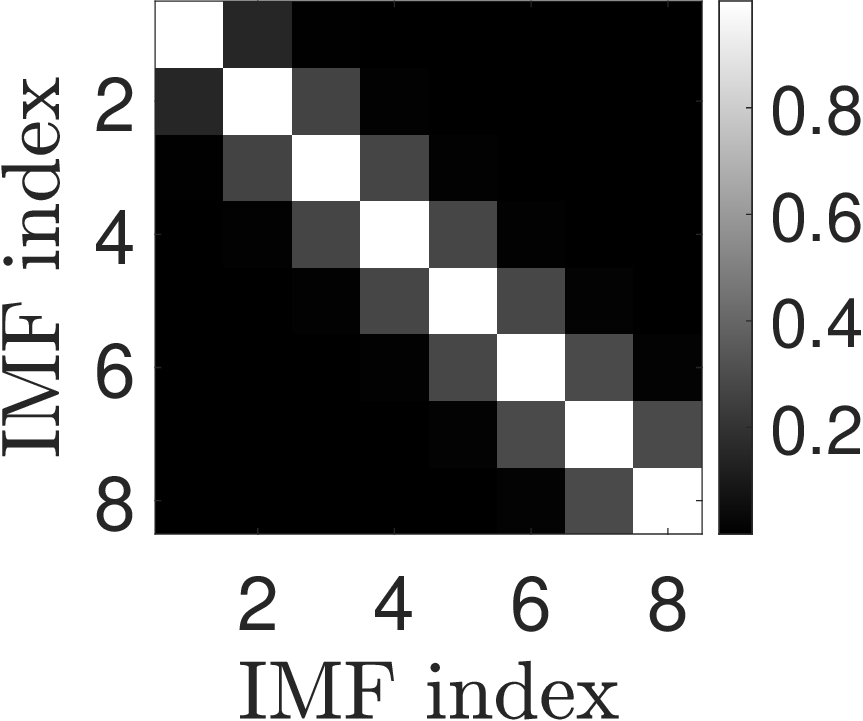} }}~\subfloat{{\includegraphics[width=0.32\textwidth]{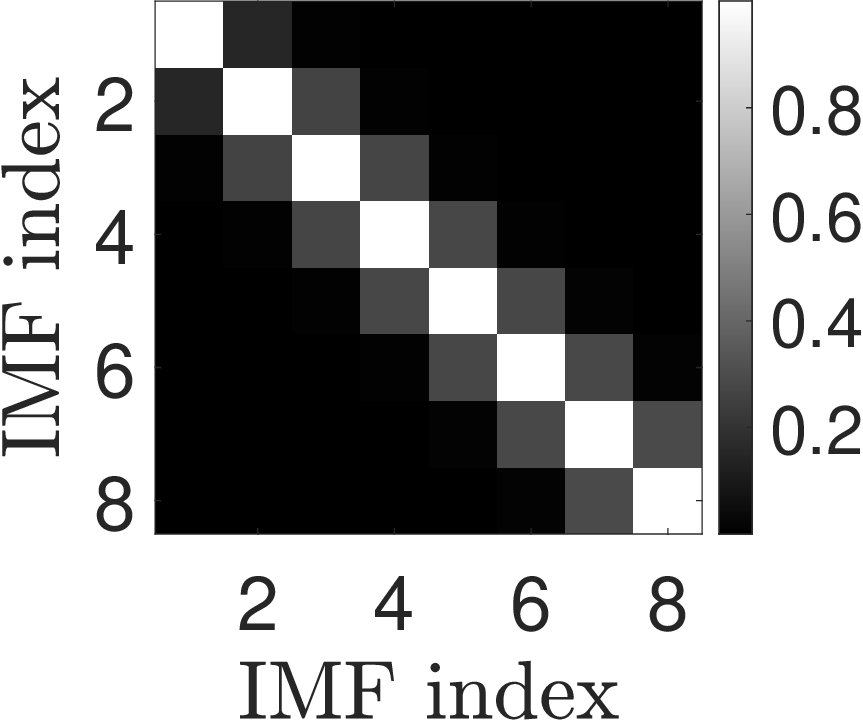} }}~\subfloat{{\includegraphics[width=0.32\textwidth]{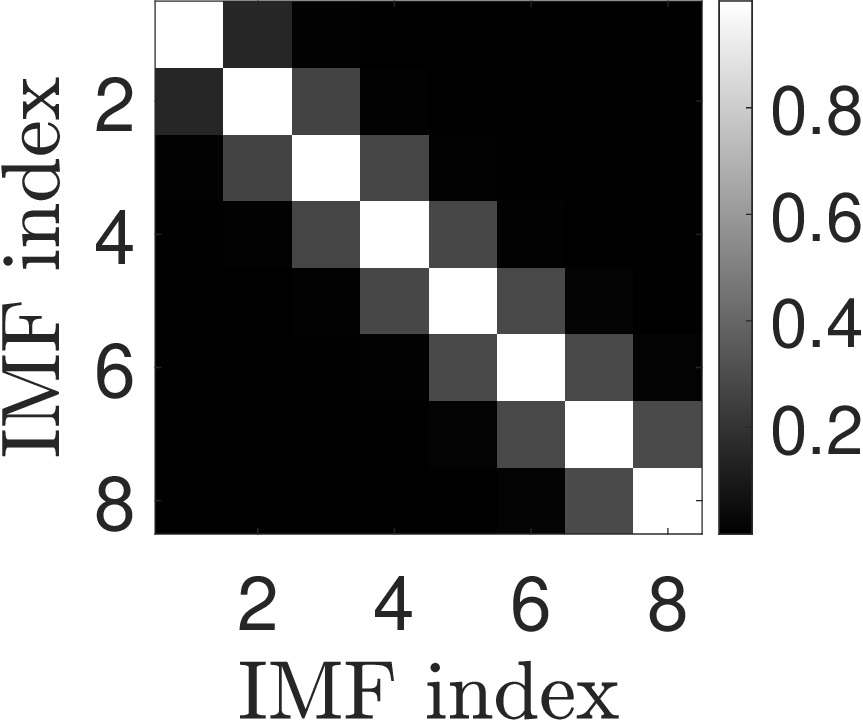} }}\\
	\caption{MvFIF decomposition filterbank plots, top row, and the correlation coefficient matrices, bottom row, of the perturbed signal for a number of channels equal to 2, 16 and 64, which are depicted in left, central and right panels respectively}\label{figS:Ex6_MvFIF}
\end{figure}

In Figure \ref{figS:Ex6_MvFIF} we plot the resulting filterbank plots and the correlation coefficient matrices. From the top row plots we can see that the filterbank and frequency--alignment properties of the MvFIF decompositions
are preserved as the number of channels are increased. Similarly the quasi--orthogonality property of the MvFIF IMFs remains unaltered even when the number of input channels increases, as shown by the diagonal structure of all the correlation plots plotted in the second row of Figure \ref{figS:Ex6_MvFIF}.

\section{Real Life Examples}\label{sec:Examples}

To show the abilities of the MvFIF algorithm in a real life setting we apply it to an EEG recording and a geophysical data set.

\subsection{EEG recording}

In this example we consider an electroencephalogram (EEG) signal of an alcoholic subject. This data arises from a large study to examine EEG correlates of genetic predisposition to alcoholism\footnote{Data are made publicly available at \url{http://kdd.ics.uci.edu/databases/eeg/eeg.html}}. It contains measurements from 64 electrodes placed on the scalp sampled at 256 Hz (3.9-msec epoch) for 1 second.
There were 122 subjects and each subject completed 120 trials where different stimuli were shown. The electrode positions were located at standard sites (Standard Electrode Position Nomenclature, American Electroencephalographic Association 1990). Further details in the data collection process can be found in \cite{zhang1995event}.

\begin{figure}%
	\centering
	\subfloat{{\includegraphics[width=0.495\textwidth]{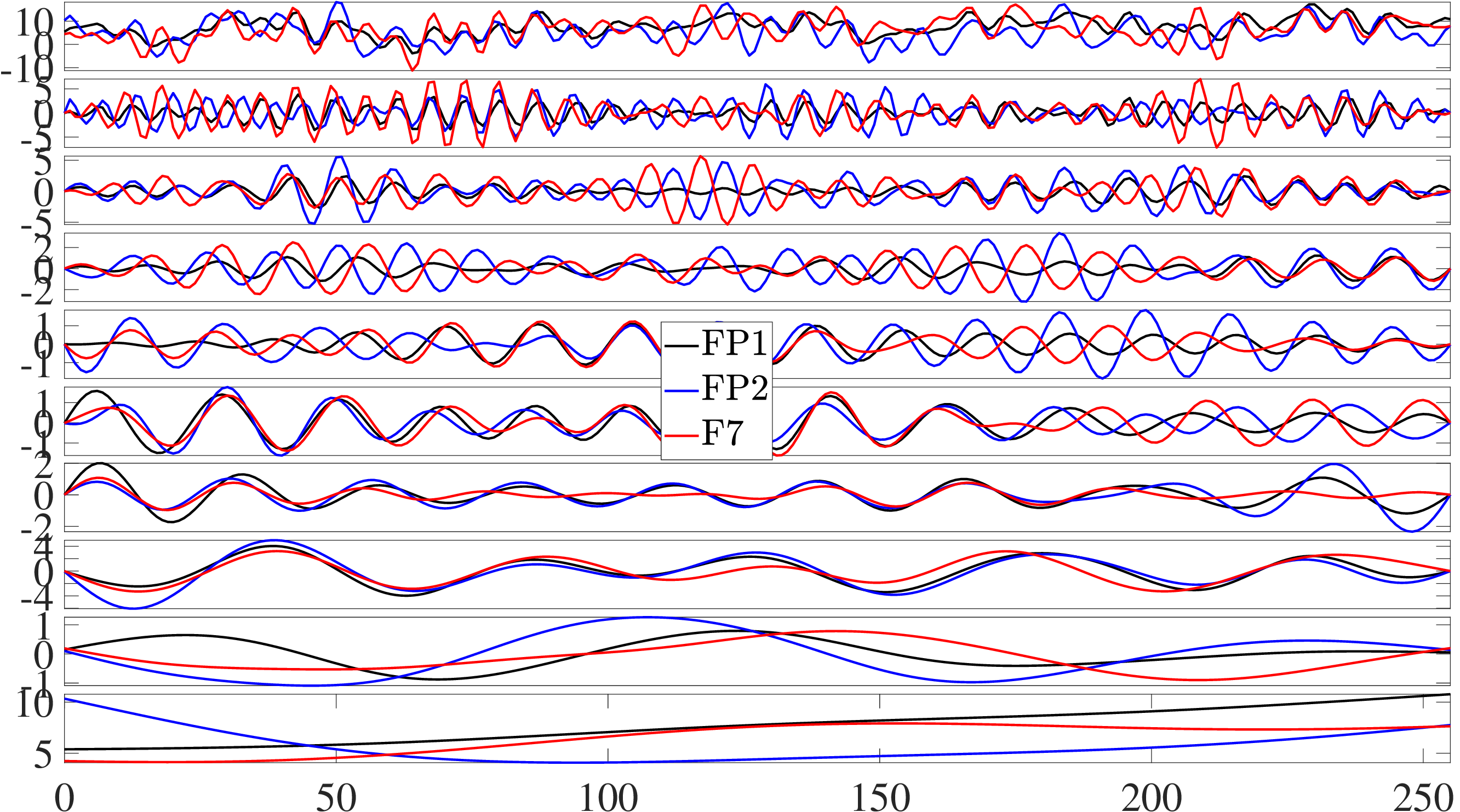} }}~\subfloat{{\includegraphics[width=0.495\textwidth]{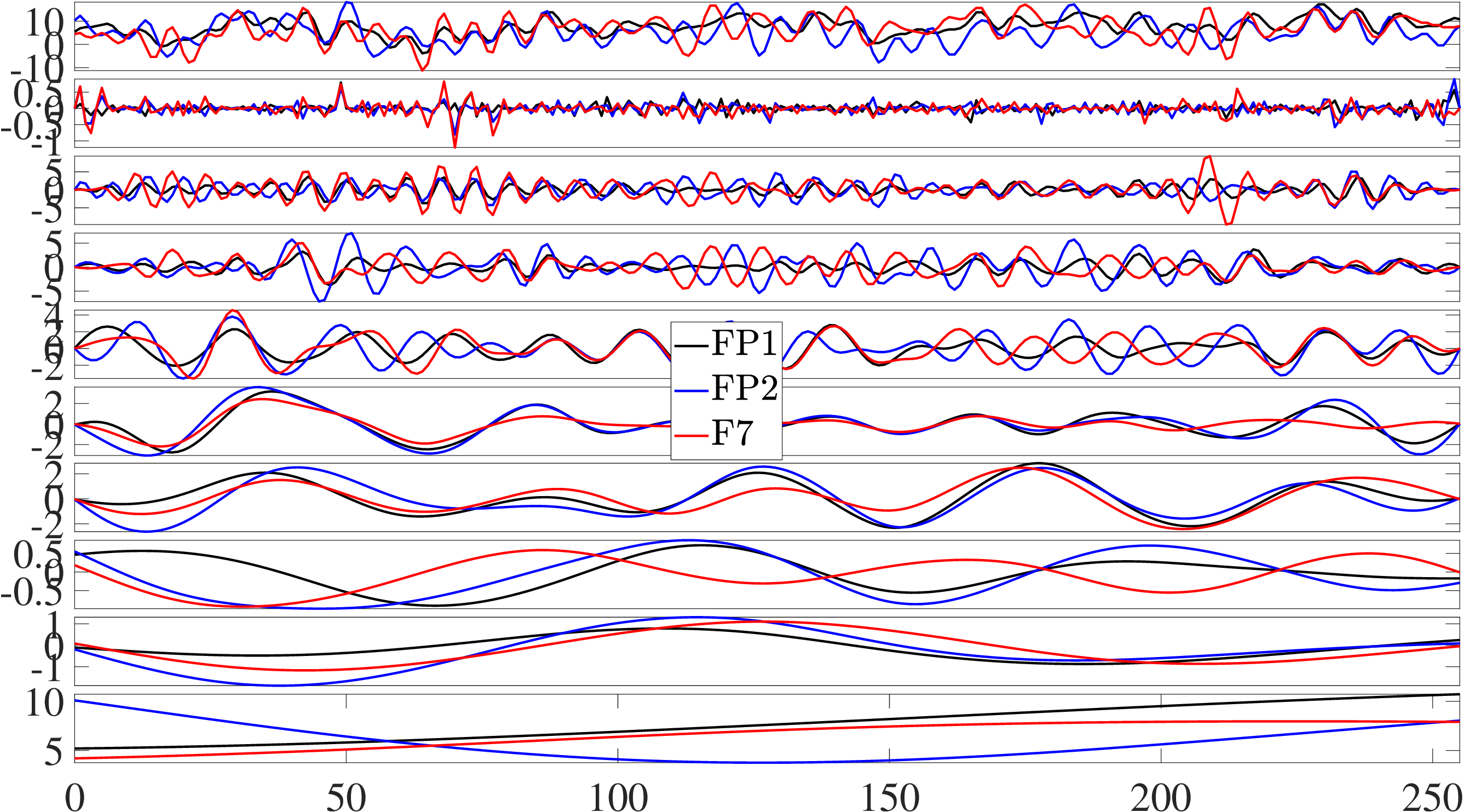} }}\\
	\subfloat{{\includegraphics[width=0.495\textwidth]{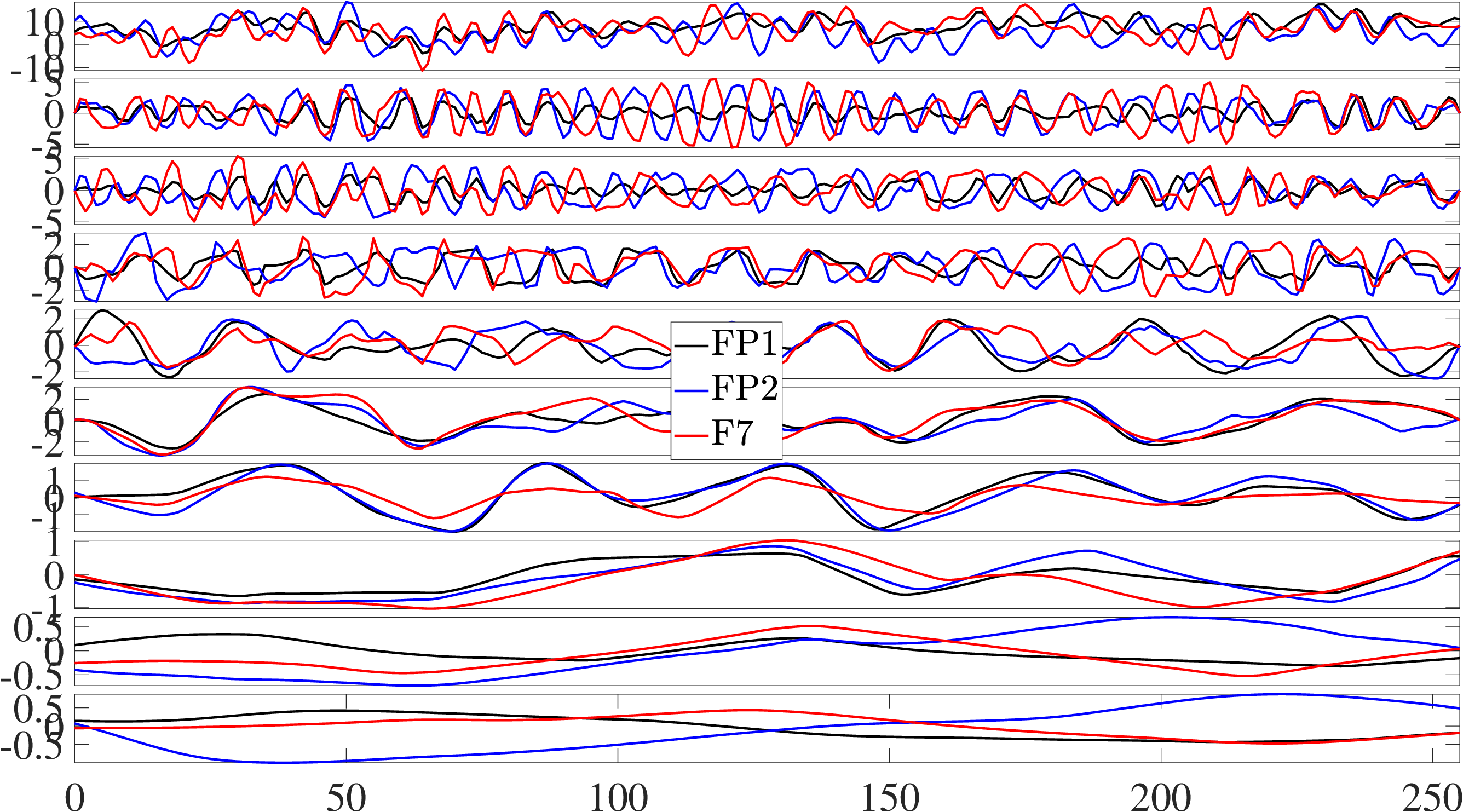} }}~\subfloat{{\includegraphics[width=0.495\textwidth]{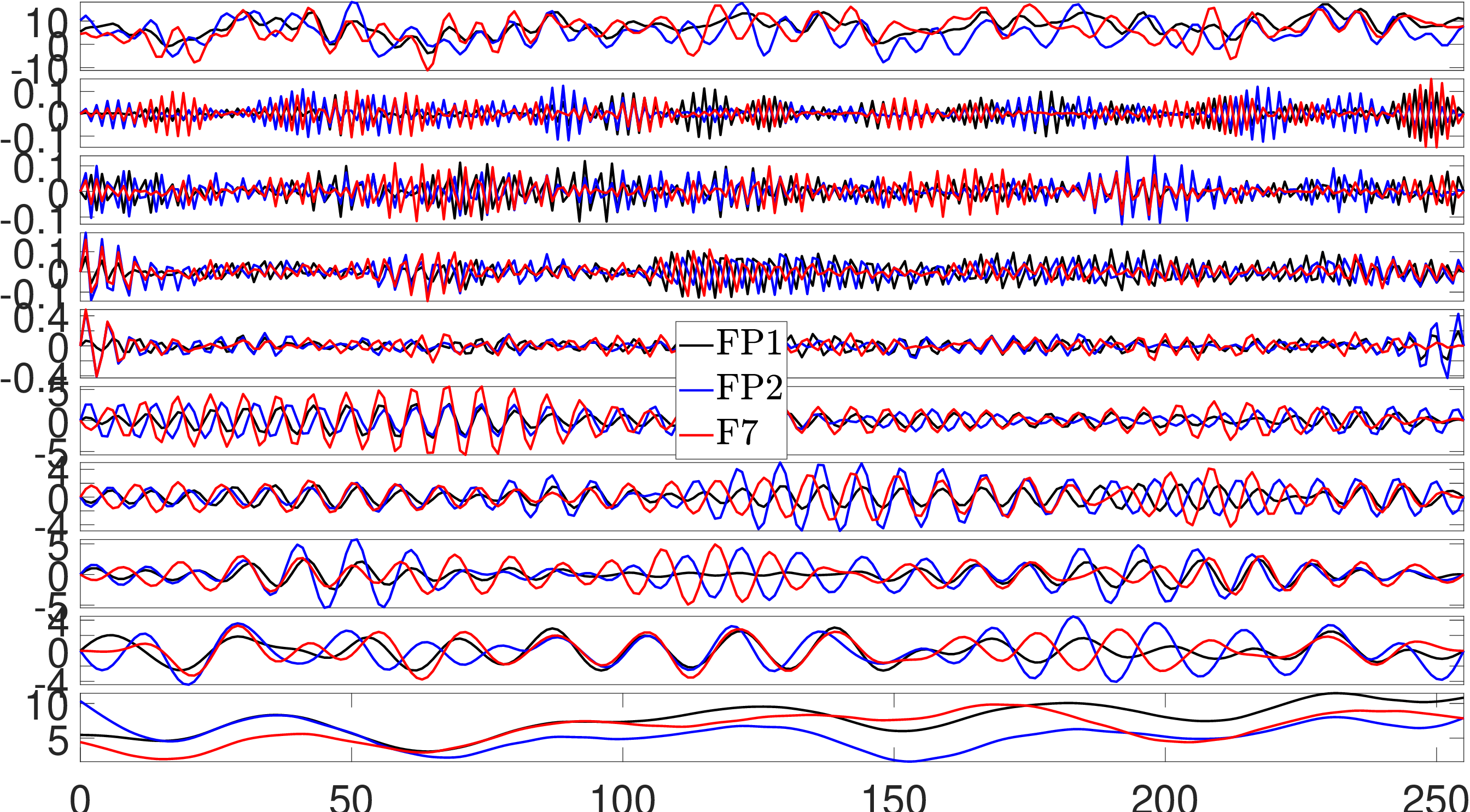} }}
	\caption{Decomposition of the first three channels of the EEG signal produced by, from left to right and from top to bottom, MvFIF, MEMD, FA-MVEMD, MVMD methods}\label{fig:Ex7_IMFs}
\end{figure}

\begin{center}
\begin{tabular}{c||c|c}
  $\#$ of Channels  & 32 &  64 \\
  \hline
  \hline
  MEMD & 20.84 s & NA \\
  FA-MVEMD & 28.66 s & 64.24 s\\
  MVMD & 13.46 s & 20.51 s \\
  MvFIF & \textbf{0.06} s & \textbf{0.38} s\\
\end{tabular}
\captionof{table}{Computational times, in seconds, versus the number of channels for different methods when the signal is decomposed in 8 IMFs plus a trend. MEMD method in its current implementation supports up to 32 channels.}\label{tab:Ex7}
\end{center}

In this work we consider one patient trial, the first 3 channels out of 64 are plotted in the top row of each panel of Figure \ref{fig:Ex7_IMFs}, and we decompose the EEG multivariate signal using MvFIF, MEMD, FA-MVEMD, and MVMD methods. A comparison of their computational time is reported in Table \ref{tab:Ex7}. From this table, it is evident that the proposed technique is extremely faster than all the other methods. From a decomposition prospective, results are shown in Figure \ref{fig:Ex7_IMFs}, where the MvFIF IMFs are depicted in the top left panel. Clearly the proposed method allows to produce comparable results with the standard MEMD, top right panel in Figure\ref{fig:Ex7_IMFs}. Whereas the FA-MVEMD algorithm, bottom panel of Figure \ref{fig:Ex7_IMFs}, appears to have some sort of instability, which is most evident in the trend behavior. Finally the MVMD technique, plotted in the bottom right panel of Figure\ref{fig:Ex7_IMFs}, does not produce the expected trend, even if we tune its parameters.

\subsection{Real Life Example: Geophysical application}\label{secS:Geo}

We consider, as a final example, the decomposition of the Earth magnetic field measurements made by one of the three satellites of the European Space Agency Swarm mission\footnote{\url{http://earth.esa.int/swarm}} from April 21 to 22, 2004. In Figure \ref{figS:Ex8_IMFs} top rows of each panel, we plot the three channels of the magnetic field, namely H, D and Z components. Each channel contains 4320 sample points.

\begin{center}
	\begin{tabular}{c||c|c|c|c}
		& MEMD & FA-MVEMD & MVMD & MvFIF  \\
		\hline
		\hline
		Time & 45.49 s & 89.08 s & 12.56 s & \textbf{0.15} s
	\end{tabular}
	\captionof{table}{Computational times, in seconds, for different methods when the signal is decomposed in 7 IMFs plus a trend}\label{tabS:Ex_8}
\end{center}

We decompose this trivariate signal using MvFIF, MEMD, FA-MVEMD, and MVMD methods. A comparison of their computational time is reported in Table \ref{tabS:Ex_8}. Also in this case the proposed technique proves to be at least two orders of magnitudes faster than any other method. The decompositions are plotted in Figure \ref{figS:Ex8_IMFs}, where the MvFIF IMFs are depicted in the top left panel. MvFIF produces comparable results both with the standard MEMD, top right panel, and the MVMD technique, plotted in the bottom right panel of Figure \ref{figS:Ex8_IMFs}. Also in this example the FA-MVEMD algorithm, bottom left panel of Figure \ref{figS:Ex8_IMFs}, appears to have some instability issues, see the trend component, and the quality of the reconstructed IMFs appears to be pretty low in general, see for instance IMF 2 and 3.

\begin{figure}%
	\centering
	\subfloat{{\includegraphics[width=0.495\textwidth]{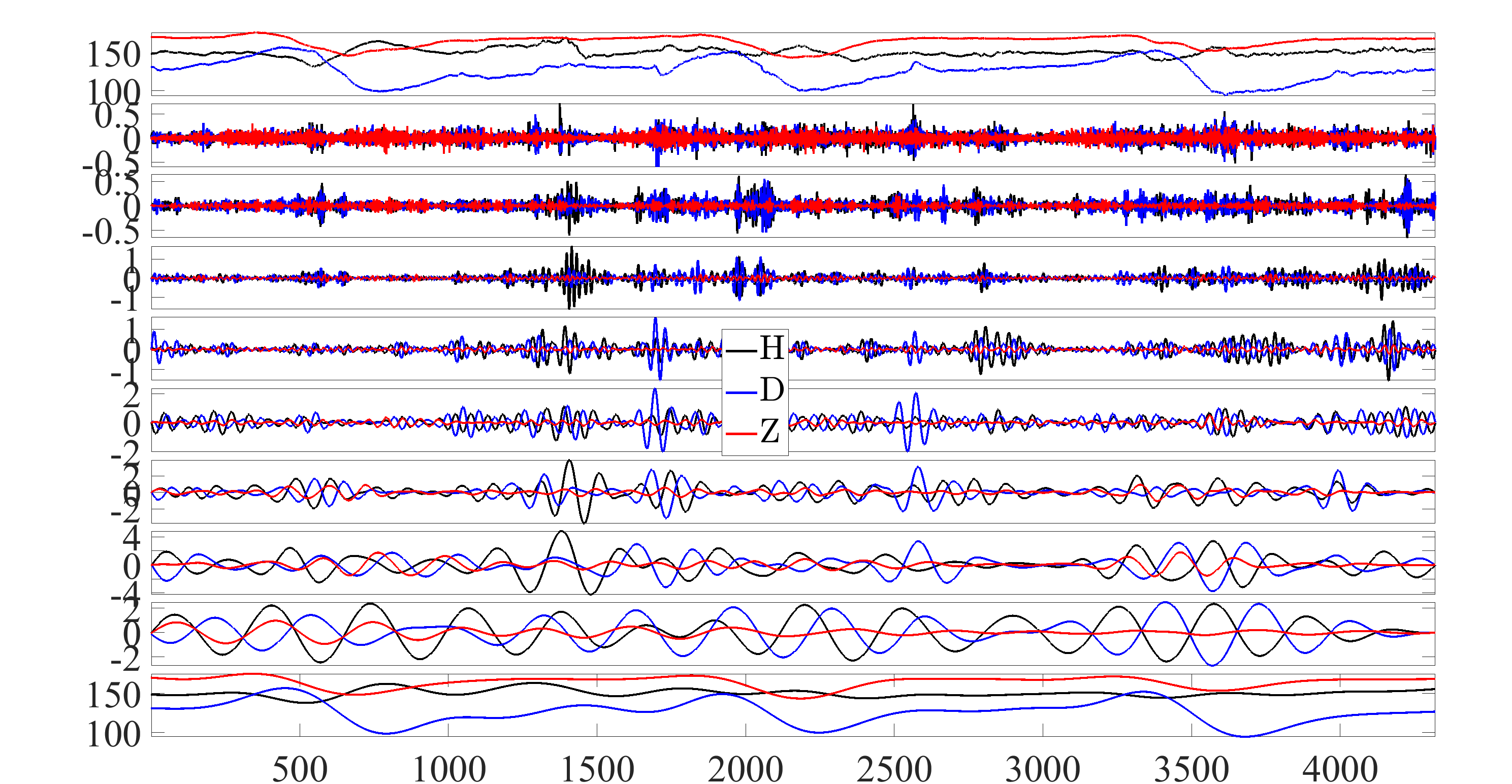} }}~\subfloat{{\includegraphics[width=0.495\textwidth]{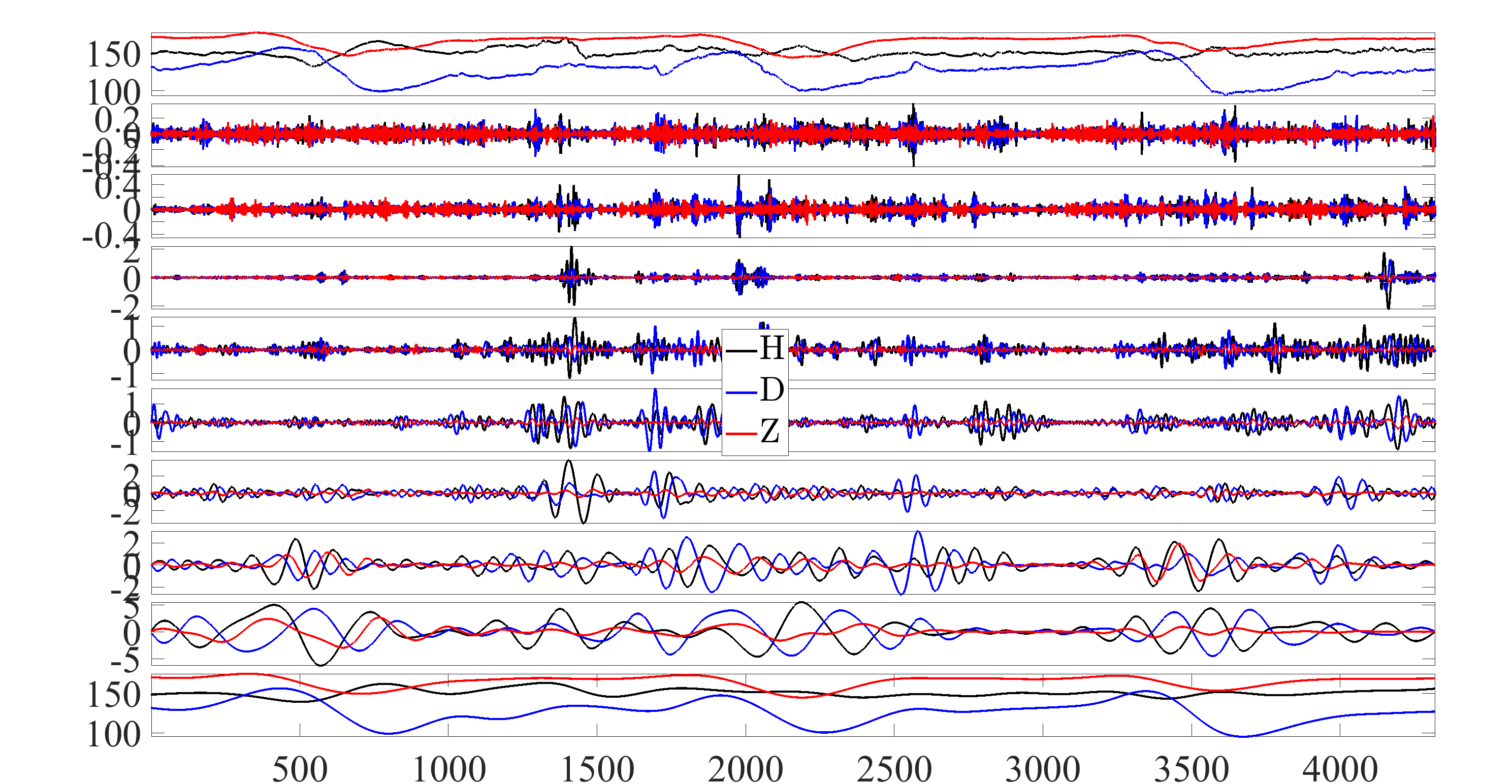} }}\\
	\subfloat{{\includegraphics[width=0.495\textwidth]{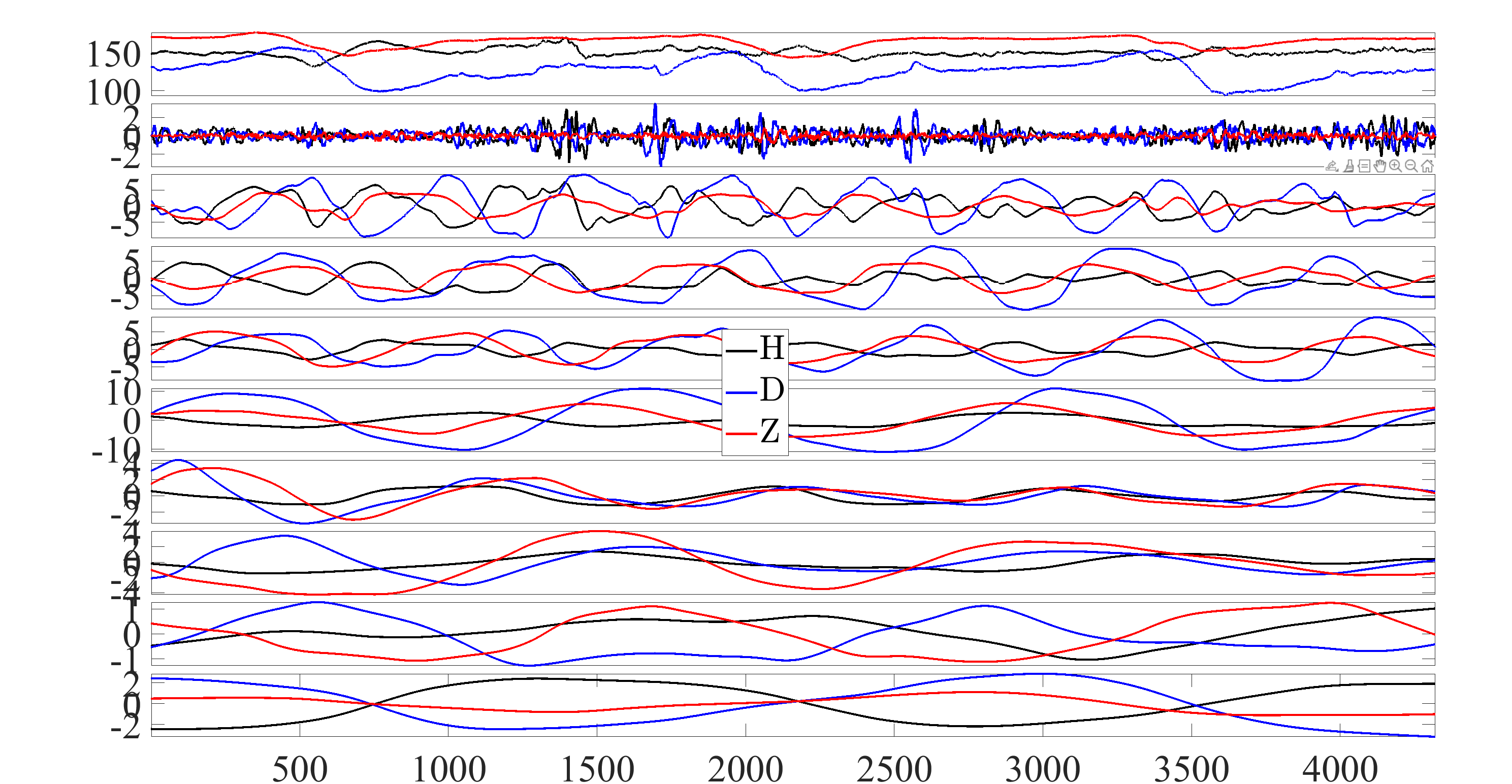} }}~\subfloat{{\includegraphics[width=0.495\textwidth]{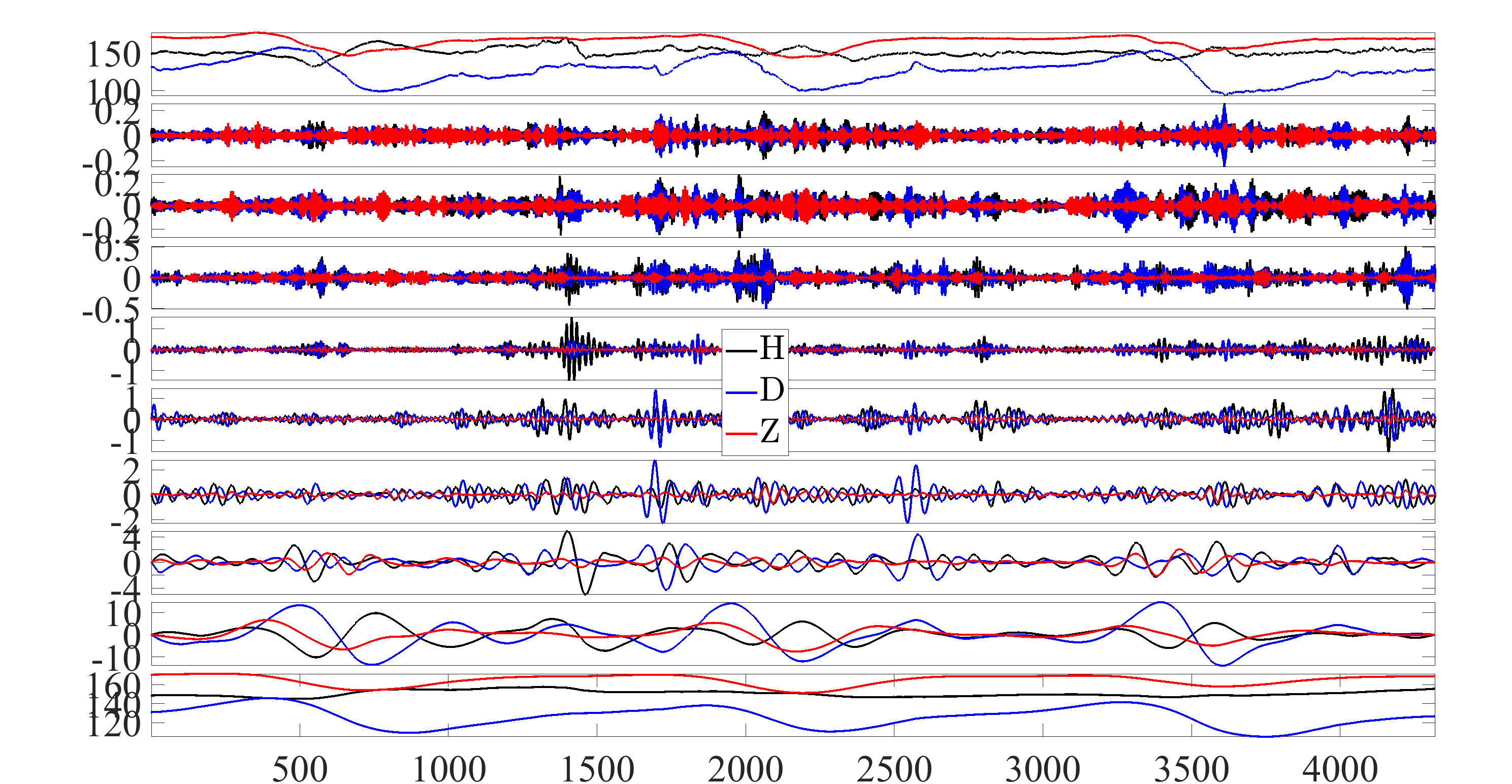} }}
	\caption{Decomposition of the magnetic field measured by one of the three ESA Swarm satellites from April 21 to 22, 2004. From left to right and from top to bottom, the MvFIF, MEMD, FA-MVEMD, and MVMD decompositions, respectively}\label{figS:Ex8_IMFs}
\end{figure}

\section{Conclusions}\label{sec:End}

Many real life problems require the time--frequency analysis of multivariate/multi-channel nonstationary signals. 

Several techniques have been proposed in recent years in the literature, 
each of them presenting a certain degree of uncertainty. 

In this paper we proposed a new algorithm, named Multivariate Fast Iterative Filtering (MvFIF), for the decomposition of multivariate/multi-channel nonstationary data sets. MvFIF proves to produce the decomposition of a multichannel signal in a fast, reliable, and certain way. In fact, the proposed approach does not require any a priori assumption on the signal under investigation, like the selection of a basis or the number of components to be extracted, and it does not necessitate of any projection for the identification of the signal moving average. 

In this work, we showed that the proposed MvFIF method exhibits the so called frequency--alignment property, i.e. the alignment of similar frequencies in each simple component extracted across all data channels. We demonstrate MvFIF ability to separate multivariate modulated oscillations using a variety of multivariate signals ranging from synthetic nonstationary signals to multivariate wGn data and real world data sets coming from electroencephalogram (EEG) measurements and the Earth magnetic field data. We also confirm that a direct application channel--wise of the standard FIF proves to be unable to guarantee frequency--alignment, as described in detail in Section \ref{sec:freq-align}.

Furthermore, we proved that MvFIF inherits all the nice features of the FIF method, including its convergence, the ability to produce quasi--dyadic filterbank decompositions, and quasi--orthogonal IMFs, and to be robust to noise, even as the number of channels increases.

From the filterbank property analysis we observed that the MvFIF first IMF tends to contain a wide interval of frequencies. This is probably due to the FIF mask length tuning parameter selection whose influence in the decomposition has never been systematically studied so far. We plan to tackle this analysis, together with the analysis of the filter shape influence in the decomposition, in a future work.

When applied to real life signals, besides confirming all the aforementioned good performance, the  MvFIF technique proved to produce results comparable with other methods proposed so far in the literature.

From a computational time point of view, the proposed method proved to be from two to three orders of magnitudes faster than other methods, as confirmed by all the numerical tests presented in this work.

The main problem with FIF, and consequently with MvFIF, is its rigidity in extracting strongly nonstationary monocomponents, like chirps, multipaths and whistles, which can be contained in a signal. This limitation is due to the so called mode--splitting problem which effects also EMD and derived algorithms \cite{yeh2010complementary}. For this reason, the Adaptive Local Iterative Filtering (ALIF) and Resampled Iterative Filtering (RIF) algorithms have been recently proposed in the literature \cite{cicone2016adaptive,barbarino2021stabilization} as a more flexible generalizations of FIF. We plan to work to the extension of these techniques to handle multivariate signals in the future.

\section*{Acknowledgment}

We thank Henri Begleiter (Neurodynamics Laboratory at the State University of New York Health Center at Brooklyn) for making publicly available the EEG recordings of alcoholic subjects we used in our tests.

We thank the European Space Agency (ESA) that supports the Swarm mission for making available the data collected by Swarm constellation satellites \url{http://earth.esa.int/swarm}.

The authors are members of the Italian ``Gruppo Nazionale di Calcolo Scientifico'' (GNCS) of the Istituto Nazionale di Alta Matematica ``Francesco Severi'' (INdAM). A.C. thanks the Italian Space Agency for the financial support under the contract ASI ``LIMADOU scienza+'' n$^{\circ}$ 2016-16-H0.

\end{document}